\documentclass[11pt,reqno]{amsart}

\usepackage{And-Cha-Gra}

\title[Shock formation in 1D conservation laws: inviscid structure]{Shock formation in 1D conservation laws I:\\Inviscid structure}

\author{John Anderson}
\address{JA: Department of Mathematics, Stony Brook University, Stony Brook, NY 11794, USA}
\email{\href{mailto:jrlanderson@math.stonybrook.edu}{\tt jrlanderson@math.stonybrook.edu}}

\author{Sanchit Chaturvedi}
\address{SC: Courant Institute of Mathematical Sciences, New York University, New York, NY 10012, USA}
\email{\href{mailto:chaturvedisanchit@nyu.edu}{\tt chaturvedisanchit@nyu.edu}}

\author{Cole Graham}
\address{CG: Department of Mathematics, University of Wisconsin--Madison, Madison, WI 53706, USA}
\email{\href{mailto:graham@math.wisc.edu}{\tt graham@math.wisc.edu}}

\date{\today}

\begin{document}

\begin{abstract}
  We study the stability and structure of shock formation in 1D hyperbolic conservation laws.
  We show that shock formation is stable near shocking simple waves: perturbations form a shock nearby in spacetime.
  We also characterize the boundary of the classical development in a spacetime neighborhood of the first time singularity.
  Finally, we describe the precise nature of nondegenerate shock formation through an expansion in homogeneous functions of fractional degree.
  We use these results in a companion paper to study the vanishing viscosity limit near shock formation.
\end{abstract}

\maketitle

\section{Introduction}

We study a hyperbolic system of conservation laws in one spatial dimension:
\begin{equation}
  \label{eq:main}
  \partial_t \psi + A(\psi) \partial_x \psi = 0, \quad t,x \in \R.
\end{equation}
The solution $\psi$ is vector-valued with $N$ components and $A$ is an $N \times N$ advection matrix.
Such conservation laws are ubiquitous in nature and describe systems spanning compressible gases, optical crystals, and elastic solids.
Initially smooth solutions of \eqref{eq:main} often develop shock singularities of great significance in applications.
Here, we study the precise nature of the onset of shocks.

\subsection{Motivation}

While the hyperbolic system \eqref{eq:main} readily forms shocks~\cite{Lax64,John74}, viscous regularizations of \eqref{eq:main} can prevent singularities.
Consider an equation of the form
\begin{equation}
  \label{eq:viscous}
  \rd_t\psi^{(\nu)}+A(\psi^{(\nu)})\rd_x \psi^{(\nu)} = \nu\partial_x[B(\psi^{(\nu)})\partial_x\psi^{(\nu)}]
\end{equation}
with viscosity $\nu > 0$, such as the compressible Navier--Stokes equations.
Under certain conditions, the parabolic regularization in \eqref{eq:viscous} can lead to global well-posedness~\cite{Kawashima87,MelVas}.
This qualitative difference between the inviscid and viscous models sharpens interest in the vanishing viscosity limit $\nu \searrow 0$, in which we formally recover the inviscid solution $\psi = \psi^{(0)}$ of \eqref{eq:main}.
In a companion work~\cite{AndChaGra25b}, we study the vanishing viscosity limit in strong norms near shock formation, where this difference first manifests.
Our analysis relies on the detailed description of \emph{inviscid} shock formation developed here.

In scalar ($N = 1$) and $2\times 2$ systems, the existence of Riemann invariants permitted Riemann~\cite{Riemann} and Lax~\cite{Lax64} to prove shock formation for a large class of initial conditions.
For larger systems $(N \geq 3)$, the absence of such invariants has largely confined the study of singularity formation to small data~\cite{John74,Sideris85,Alinhac01, Alinhac02, Alinhac03,CP16}.
We are interested in the extension of such results to wider classes of initial data.
Inspired by work on the compressible Euler equations in multiple dimensions~\cite{SHLW,LukSpe01,LukSpe02}, we show that solutions near so-called simple waves form shocks.
Though still perturbative, these data can be arbitrarily large in amplitude, and thus lie well outside the small-data setting.

Plasmas, elastic materials, and crystals are known to form shocks along waves of intermediate speed.
It is thus of physical interest to prove that such shocks can form from an open set of initial data.
Intermediate waves pose special challenges, and in fact interactions between intermediate and faster and slower waves have the potential to destabilize large simple waves.
Nonetheless, under a very mild nondegeneracy assumption, we show that intermediate waves form shocks in a stable manner.
This additional difficulty is a feature of the large data considered here; it does not arise in the small-data setting.

Finally, due to finite speed of propagation, the classical evolution of hyperbolic PDEs need not stop at the first time of singularity.
Recently, the maximal domain of classical evolution, or \emph{maximal development}, has attracted much attention.
We describe the maximal development of \eqref{eq:main} in a spacetime neighborhood of shock formation under no genericity assumptions.

\subsection{Setup}

We now describe our equation and solutions in detail.
The solution $\psi$ takes values in an $N$-dimensional vector space $\m{V} \cong \R^N$ that we term the ``state space.''
It represents a collection of physical variables such as mass and momentum density.
For the advection, we assume:
\begin{enumerate}[label={(H\arabic*)},series=hyp,itemsep=3pt]
\item
  \label{hyp:smooth-A}
  $A \colon \m{V} \to \m{V} \otimes \m{V}^*$ is smooth.

\item
  \label{hyp:strict-hyp}
  $A$ has $N$ distinct eigenvalues $\lambda_{(I)}$ with right eigenvectors $r_I$.

\item
  \label{hyp:gen-non}
  Some eigenvalue $\lambda_{(I_0)}$ is genuinely nonlinear, meaning $\inf \Der \lambda_{(I_0)} \cdot r_{I_0} > 0$.
\end{enumerate}
We do \emph{not} require $A$ to be the derivative of a flux function, so our analysis extends to nonlinear advection equations \eqref{eq:main} that are not conservative in nature.
Nonetheless, for the sake of familiarity we refer to \eqref{eq:main} as a conservation law, and we are primarily motivated by physical systems with true conservative structure.

Our analysis is dominated by the genuinely nonlinear $I_0$-characteristic, so \ref{hyp:strict-hyp} could be relaxed to the hypothesis that $I_0$ alone has multiplicity one.
The nonshocking characteristics could have higher and variable multiplicity, but we work in the strictly hyperbolic framework \ref{hyp:strict-hyp} for simplicity.

The genuine nonlinearity condition \ref{hyp:gen-non} was introduced by Lax~\cite{Lax64} and is intimately connected with shock formation.
The eigenvalue $\lambda_{(I_0)}$ is the speed of the $I_0$-characteristic.
Genuine nonlinearity allows fast portions of the solution to overtake slow, causing a shock.

Hypotheses~\ref{hyp:smooth-A}--\ref{hyp:gen-non} constitute the framework studied by John~\cite{John74} with small data.
Here, we consider perturbations of (potentially large) ``simple waves.''
\begin{definition}
  A solution $\Theta$ of \eqref{eq:main} is a \emph{simple wave} associated to the eigenvalue $\lambda_{(I_0)}$ if it is constant along the $I_0$ characteristic.
  That is, $\partial_t \Theta + \lambda_{I_0}(\Theta) \partial_x \Theta = 0.$
\end{definition}
For more details on the construction and nature of simple waves, see Section~\ref{sub:simple-waves} below.
We perturb around a simple shocking wave $\Theta$:
\begin{enumerate}[hyp]
\item
  \label{hyp:simple}
  $\Theta$ is a simple wave associated to $\lambda_{(I_0)}$ that is initially smooth but forms a shock in finite time.
\end{enumerate}

This wave represents the ``background'' shocking behavior, and we are naturally interested in its stability.
We find that this depends on whether $\lambda_{(I_0)}$ is an extremal or intermediate speed.
For intermediate shocks, we need $\Theta$ to be nondegenerate in a very mild sense; see Condition~\ref{cond:L1Th} below.
No such condition is needed when $\lambda_{(I_0)}$ is extremal.
For more details on this technical consideration, see Sections~\ref{sec:proof-over} and \ref{sec:reneq}.

Under these weak conditions, we show that shock formation is stable (Theorem~\ref{thm:hyperbolic}).
If $\Theta$ is nondegenerate in a stronger sense (Condition~\ref{cond:genericity}), we can further describe $\psi$ to arbitrary order near shock formation (Theorem~\ref{thm:homexprough}).

\subsection{Main results}

We now state informal versions of our main results.
\begin{theorem}
  \label{thm:hyperbolic}
  Assume \ref{hyp:smooth-A}--\ref{hyp:simple}, $N \geq 2$, and $\Theta$ is either extremal or mildly nondegenerate.
  Let $\psi(0, \anon)$ be smooth and $\m{C}^1$-close to $\Theta(0, \anon)$.
  Then $\psi$ develops a shock near that of $\Theta$ and remains uniformly smooth in $t$ and the $I_0$-eikonal function up to its shock time.
  Moreover, we can classify the boundary of the maximal globally hyperbolic development in a spacetime neighborhood of the first singularity.
\end{theorem}
\noindent
For a precise statement, see Theorem~\ref{thm:SF} below.

To show Theorem~\ref{thm:hyperbolic}, we change to a geometric coordinate system in which $\psi$ remains smooth.
This is based on the widely-used eikonal function, defined in \eqref{eq:eikonal} below.
The shock blowup is encoded in the change between physical and eikonal coordinates, which necessarily degenerates.
We use geometric methods pioneered by Christodoulou~\cite{Christodoulou01} to prove this degeneracy.
The resulting estimates also allow us to characterize the maximal globally hyperbolic development (MGHD) near shock formation.
For a description of MGHDs, see Appendix~\ref{sec:appendix1}.

Under additional nondegeneracy, Theorem~\ref{thm:hyperbolic} allows us to develop a complete homogeneous expansion for $\psi$ about shock formation.
\begin{theorem}
  \label{thm:homexprough}
  Assuming further nondegeneracy, up to the shock time, $\psi$ admits an expansion in homogeneous functions near the first singularity.
  Modulo a Galilean transformation, the leading part is an inverse cubic $\cub$ in the shocking component alone solving $x = -at \cub + b \cub^3$ for some $a,b \neq 0$ of the same sign.
\end{theorem}
\noindent
See Theorem~\ref{thm:psi-hom-exp} for a precise statement.

The self-similar profile $\cub$ also arises in generic shock formation in the Burgers equation.
In~\cite{CG23}, the second and third authors used an analogous expansion to describe the vanishing viscosity limit of Burgers in strong norms near shock formation.
In this project, we do the same for systems of conservation laws.
The present paper contains our inviscid analysis, and our companion paper~\cite{AndChaGra25b} uses Theorem~\ref{thm:homexprough} to treat the vanishing viscosity limit of \eqref{eq:viscous}.

We have assumed for convenience that $A$, $\Theta(0,\anon)$, and $\psi(0,\anon)$ are $\m{C}^\infty$.
In reality, Theorem~\ref{thm:hyperbolic} only requires some finite degree of regularity.
Similarly, if our inputs are $\m{C}^k$ for large $k$, the expansion in Theorem~\ref{thm:homexprough} will hold to some large degree of approximation depending on $k$.
We do not seek to quantify or optimize this dependence, though it can in principle be determined from our proofs.

Likewise, as mentioned above, our assumption \ref{hyp:strict-hyp} of strict hyperbolicity is largely a matter of convenience.
We are strongly motivated by the compressible Euler equations, which satisfy \ref{hyp:strict-hyp}.
Nonetheless, an examination of the proof shows that we only make use of the simplicity of the shocking eigenvalue $\lambda_{(I_0)}$.
Other eigenvalues can have higher or variable multiplicity.

If the shocking eigenvalue has higher (constant) multiplicity, some of the arguments here and in \cite{AndChaGra25b} carry through, and others do not.
Shocks with multiplicity do arise in certain models of magnetohydrodynamics and elasticity, and such equations have been the subject of recent investigation~\cite{AnChe21,AnCheYin22a,AnCheYin22b,AnCheYin23}.
The extension of our efforts to shocks with multiplicity strikes us as an interesting direction for future work.

Finally, our results are perturbative around simple waves, but we expect this setting to be generic.
Namely, we speculate that typical first singularities in strictly hyperbolic, genuinely nonlinear conservation laws are isolated shocks forming along a single characteristic.
(The preceding conditions preclude vacuum formation in fluids, for example.)
In this case, shortly before shock formation, a rescaling of the solution would be close to a simple wave, and would thus satisfy our hypotheses.
We leave this exploration of generic large data to future inquiry.

\subsection{Related works}

Our study of shock formation joins a vast literature on singularities in conservation laws.
Here we describe a number of related efforts, though we cannot be exhaustive.
We focus on the inviscid theory, as our companion paper~\cite{AndChaGra25b} discusses the literature on regularization and vanishing viscosity.
For an overview of hyperbolic conservation laws and shocks, we direct the reader to Dafermos \cite{Dafermos}, Liu \cite{Liu_2021}, and the references therein.

\subsubsection{Well-posedness of conservation laws}
Due to the possibility of singularity formation from smooth data, conservation laws are often studied in a weak framework.
In celebrated work, Glimm proved global existence for solutions of 1D genuinely nonlinear conservation laws with small-BV data~\cite{Glimm65}.
This compactness construction inspired numerous advances \cite{DiPerna76, Bressan92, BaiJen98, Risebro93}, but does not guarantee uniqueness.
Nevertheless, uniqueness is now known thanks to work of Bressan and LeFloch \cite{Bressan95,BreLeF97,Bressan00}.

In higher dimensions there is no known global theory, and in fact BV does not provide a good framework if used the most direct sense~\cite{rauch1986bv}.
We are still at the stage of providing examples of global configurations with shocks, as in very interesting recent work of Ginsberg and Rodnianski \cite{GinRod24}.

\subsubsection{Shock formation in $1$D}
In $1 + 1$ settings, shock formation goes back to Riemann~\cite{Riemann}.
Building on Riemann invariants, Lax~\cite{Lax64} introduced the notion of genuine nonlinearity and established singularity formation for such $2\times 2$ systems from a large class of initial data.
In his classical work, John~\cite{John74} proved small-data singularity formation for larger genuinely nonlinear systems.
Other recent works in $1 + 1$ include \cite{Liu79,CP16,BucDriShkVic_2021,NSV23,NeaRicShkVic23,NeaShkVic24}.
Our setting is perhaps closest to that of John~\cite{John74} and Christodoulou and Perez~\cite{CP16}.
We note that a key idea in some modern approaches to shock formation originated in work of Keller and Ting \cite{KelTin66}, who studied the density of characteristic curves.
We discuss the importance of this quantity in Section~\ref{sec:proof-over}.

An, Chen, and Yin have recently studied $1 + 1$ shocks in the absence of strict hyperbolicity~\cite{AnChe21,AnCheYin22a,AnCheYin22b,AnCheYin23}.
These works also develop shock formation into ill-posedness in certain spaces, using a connection first explored by Granowski~\cite{Gra18}.

\subsubsection{Higher dimensional shock formation}
The formation of singularities in higher dimensions was first studied by John~\cite{Joh81} and Sideris~\cite{Sideris85}.
These initial results did not provide detailed descriptions of the nature of the singularities.
In an impressive sequence of works, Alinhac~\cite{Alinhac01, Alinhac02, Alinhac03}, showed the formation of nondegenerate shock singularities that are, roughly speaking, isolated within the constant-time hypersurface of first blowup.

The problem was revisited by Christodoulou in his landmark work~\cite{Christodoulou01} on isentropic irrotational relativistic Euler.
This developed several important geometric ideas for the study of shock formation and sparked many interesting followups~\cite{ChrMia,SHLW,Speck16,MiaYu17}.
One important breakthrough is due to Luk and Speck~\cite{LukSpe01}, who were the first to allow for additional wave speeds, studying compressible Euler with nontrivial vorticity up to the first singularity.
More recent works permitting multiple wave speeds include \cite{BucShkVic_2019, LukSpe02, AbbSpe, ShkVic_2024}.
Some of these also construct (interesting portions of) the maximal globally hyperbolic development, described further below.

The work \cite{BucShkVic_2019} is part of a remarkable program by Buckmaster, Shkoller, and Vicol to understand shock formation at the point of first singularity.
In~\cite{BucShkVic_2019, BucShkVic_2020, BucShkVic_2020.2}, the authors use modulation theory to describe shock formation in terms of a self-similar solution to the Burgers equation.
This plays a role analogous to $\cub$ in Theorem~\ref{thm:homexprough}.
In addition, Buckmaster and Iyer~\cite{BuckIye_2020} show the formation of unstable shocks in two-dimensional polytropic compressible Euler.

\subsubsection{The maximal globally hyperbolic development}
Most studies of shock formation to date are adapted to ``constant $t$'' hypersurfaces, so they only grant access to the solution up to the time this foliation first hits the singularity.
One would like to progress beyond this, to a maximal region in which the solution can be defined classically.
The combination of a maximal region and its classical solution is termed a maximal globally hyperbolic development (MGHD).
For details, see Appendix~\ref{sec:appendix1}.
Beyond their intrinsic interest, MGHDs play an important role in the shock development problem discussed below.

Christodoulou's book \cite{Christodoulou01} treats constant-$t$ hypersurfaces, but he also describes how to adapt the arguments to other hypersurfaces and regions of spacetime.
This allows him to explore the causal structure of an interesting portion (depending on finer properties of the data) of the boundary of the MGHD.
Since then, there has been much work explicitly characterizing portions of the MGHD boundary in the presence of vorticity and entropy, under suitable genericity conditions.
For a general introduction and a construction in $1 + 1$ dimensions see~\cite{AbbSpe23}.

In the very difficult higher dimensional case, we highlight two particular works.
In \cite{AbbSpe}, Abbrescia and Speck construct the entire ``crease'' (the set of preshocks) and a portion of the singular boundary in the $3 + 1$ compressible Euler equations under genericity in one direction.
And in \cite{ShkVic_2024}, Shkoller and Vicol construct the boundary of a generic MGHD for the $2 + 1$ isentropic Euler equations in a spacetime neighborhood of the first singularity, including preshocks, singular boundary, and Cauchy horizon.
We describe these different boundary types in Section~\ref{sec:SFstatements}.
We emphasize that both these exciting works allow for acoustical and transport phenomena.

\subsubsection{The germ of first singularities}
The homogeneous expansion we construct in Theorem~\ref{thm:homexprough} expresses the germ of a nascent shock.
A number of works describe the germs of singularities via expansion, including some mentioned above~\cite{Christodoulou01,BucShkVic_2019,BucShkVic_2020,BucShkVic_2020.2,BucDriShkVic_2021,LukSpe02,AbbSpe,NSV23,NeaRicShkVic23,ShkVic_2024}.
The homogeneous expansion we use comes from \cite{CG23}, and is packaged to assist the study of the vanishing viscosity limit.
This inverse-cubic behavior is the only known stable configuration in high-regularity settings, but see \cite{NeaRicShkVic23} for other examples in low regularity.

\subsubsection{Post-shock evolution}
On physical grounds, we should be able to extend solutions of conservation laws in a weak class past the moment when a shock (or other singularity) forms.
This is known for small-BV solutions in one dimension, as such weak solutions are global in time~\cite{Glimm65}.
For large data and in higher dimensions, this remains a formidable problem.

Post-shock evolution has been constructed in certain symmetry-reduced settings, where solutions are functions of two variables.
For example, see \cite{yin04,ChrLis16} for spherically symmetric Euler and \cite{BucDriShkVic_2021} for Euler under azimuthal symmetry, building on work on the $p$-system initiated by Lebaud \cite{Leb94}.
The azimuthal symmetry in \cite{BucDriShkVic_2021} is particularly challenging, as it permits nonzero vorticity.
Without symmetry, Yin and Zhu have treated scalar conservation laws in higher dimensions~\cite{YinZhu22}.
For systems, the only existing work is due to Christodoulou \cite{Chris_2019}, who studies a different weak formulation of the compressible Euler equations.
This formulation forces the flow to remain isentropic and irrotational, allowing the solution to be represented by a scalar fluid potential.
In physical fluids, shock formation generally introduces entropy and vorticity, so this  ``shock development problem'' remains a fundamental open question.

\subsubsection{Singularity formation in related models}

While we focus on inviscid equations, singularities can develop in the face of regularizing terms if the latter are sufficiently weak.
For example, much effort has focused on shock formation (or global well-posedness) in the one-dimensional Burgers equation with fractional diffusion or dispersion~\cite{KiNaSh_2008, DongDuLi_2009, AliDroVov_2007,ChiMorPan_2021,Yang_2020,yang22,OhPas_2021}.

In higher dimensions, recent breakthroughs of Merle, Rapha\'{e}l, Rodnianski, and Szeftel show that other singularities can form in the presence of traditional diffusion.
The authors construct smooth implosion solutions to the barotropic Euler equations~\cite{MeRaRoSz01} and use them to exhibit implosion singularities in 3D compressible Navier--Stokes~\cite{MeRaRoSz02} and higher-dimensional nonlinear Schr\"odinger equations~\cite{MeRaRoSz03}.
We refer the reader to \cite{TrCaGo23,cao2023non} for a different approach to an implosion construction and for non-radial implosion in compressible Euler and Navier--Stokes.
Very recently, Chen, Cialdia, Shkoller, and Vicol~\cite{CCSV} and Chen~\cite{Chen} have proved vorticity blow up for the compressible Euler equations in $2$D and $3$D, respectively.

\subsection{Organization}

We outline our strategy in Section~\ref{sec:proof-over} and introduce geometric notation in Section~\ref{sec:sfbasicsetup}.
Section~\ref{sec:STheorems} contains our rigorous formulation of shock formation and stability (Theorem~\ref{thm:SF}).
We develop convenient forms for our equations in Section~\ref{sec:reneq} and prove our essential hyperbolic estimates in Section~\ref{sec:HyperbolicEsts}.
We apply these in Section~\ref{sec:SF} to prove Theorem~\ref{thm:SF}, and thereby Theorem~\ref{thm:hyperbolic}.
In Section~\ref{sec:homogeneous}, we construct a homogeneous expansion for $\psi$ and prove Theorem~\ref{thm:homexprough}.
Appendix~\ref{sec:appendix1} reviews standard notions of causality and defines the MGHD.

\section*{Acknowledgments}
JA was partially supported by the National Science Foundation under grant DMS-2103266.
SC was supported by the Simons Foundation Award 1141490.
CG was partially supported by the National Science Foundation through the grants DMS-2103383 and DMS-2516786.

We thank Jonathan Luk for several useful discussions.
We are also grateful to Leo Abbrescia and Jan Sbierski for their insights on uniqueness of MGHDs.

\section{Proof overview}
\label{sec:proof-over}

In this section, we describe the proofs of our main results, informally stated in Theorems~\ref{thm:hyperbolic} and \ref{thm:homexprough} and rigorously expressed as Theorems~~\ref{thm:SF} and \ref{thm:psi-hom-exp}.
Our first result establishes the stability of shock formation and estimates on the solution.
Our second uses these estimates to build a homogeneous expansion for $\psi$.

Given a simple wave $\Theta$ that forms a shock in finite time (see Section~\ref{sub:simple-waves}), we wish to perturb the data and show that a shock still forms.
We do so by carefully selecting advantageous independent and dependent variables.
In brief, $\psi$ proves to be smooth in a solution-dependent set of eikonal coordinates, and a nonlinear transformation of the unknowns allows us to single out the most singular part of the solution.

\subsection{Independent variables}

Because the simple wave is confined to a single eigenspace of $A$, it essentially solves a scalar conservation law.
It thus remains constant along a family of straight lines in the $(t, x)$ plane, and shock formation corresponds to the divergence of the line density (see Section~\ref{sec:gquant}).

When we perturb $\Theta$, we hope to have a similar picture, in which $\psi$ is closely bound to a scalar function whose level sets are nearly straight.
Shock formation should then be tied to the spatial density of these level sets.
This idea dates back to Keller and Ting~\cite{KelTin66} and John~\cite{John74}.
It was refined by Christodoulou~\cite{Christodoulou01}, and we follow his approach here.

Because $\Theta$ is a simple wave associated to $\lambda_{(I_0)}$, this eigenvalue is the most important in the problem.
We introduce a scalar ``eikonal'' function $u$ that is constant along the characteristic curves of $\lambda_{(I_0)}(\psi)$.
These curves are thus the level sets of $u$, and we show that they are nearly straight, as suggested above.
The quantity $\mu = (\partial_x u)^{-1}$ measures the inverse foliation density of these curves.
Where it vanishes, the foliation density diverges, and a shock forms.

The key idea is to use $u$ as a new independent variable, so we think of solutions as functions of $(t, u)$ rather than $(t, x)$.
For example, the simple wave $\Theta$ is constant along the characteristics of $\lambda_{(I_0)}$ and initially smooth, so $\Theta$ is a smooth function of $u$ alone.
When we perturb, $\psi$ is no longer exactly constant in $u$, but we show that it remains uniformly smooth in the $(t, u)$ coordinate system.
This does not contradict shock formation, as the change of variables from $(t, u)$ to $(t, x)$ becomes singular where $\mu$ vanishes.

To avoid potential confusion, let $\tau = t$ and consider the coordinate vector fields $(\partial_\tau, \partial_u)$ associated to the independent variables $(\tau, u)$.
So $\partial_\tau$ represents the time derivative with $u$ held constant, while $\partial_t$ holds $x$ constant.
Then we can compute $\partial_\tau = \partial_t + \lambda_{(I_0)} \partial_x$ and $\partial_u = \mu \partial_x$.
We express our equations in terms of $\partial_\tau$ and $\partial_u$, and use this to prove smoothness estimates in $(t, u)$.
Crucially, we have ``renormalized'' the spatial derivative $\partial_x$ by the inverse foliation density $\mu$ to measure $\partial_u = \mu \partial_x$.
While $\partial_x \psi$ blows up at the shock, $\partial_u \psi$ does not because we have ``tamed'' it through the vanishing factor $\mu$.
In short, $\partial_x$ is not a good measure of regularity because it cuts through too many level sets of $u$ near the shock; $\partial_u$ cures this pathology.

Ultimately, we wish to show that shock forms, so $\mu \to 0$.
In the simple wave, this is straightforward: one can easily check that $\partial_\tau^2 \mu = 0$ and $\partial_\tau \mu$ is initially negative somewhere.
As a result, $\mu$ decreases linearly to zero along the characteristic line through the first shock.
Perturbing, we show that $\partial_\tau^2 \mu$ remains small, so $\mu$ still decreases to zero along some characteristic curve.
That is, $\psi$ forms a shock.

\subsection{Dependent variables}

The other important realization is that certain combinations of dependent variables behave better than others.
In $2 \times 2$ systems, the celebrated Riemann invariants (see \cite{CouFri48}) make this quite explicit.
For larger systems, we can instead differentiate \eqref{eq:main} and project $\partial_x \psi$ onto the eigenspaces of $A$.
John took this approach in his analysis of the small data problem~\cite{John74}.

If $(\partial_x \psi)^I$ represents the $r_I$-component of $\partial_x\psi$ in the eigenbasis of $A(\psi)$, we can check that
\begin{equation*}
  (\partial_t + \lambda_{(I)} \partial_x) (\partial_x \psi)^I = \m{N}^I(\partial_x \psi),
\end{equation*}
where $\m{N}^I$ represents some nonlinearity.
Genuine nonlinearity \ref{hyp:gen-non} ensures that $\m{N}^{I_0}$ has a Riccati term driving $(\partial_x \psi)^{I_0}$ to blow up.
We find that the other equations are better-behaved, and $(\partial_x \psi)^{I'}$ actually remains bounded for $I' \neq I_0$.
We describe these calculations in Section~\ref{sec:sfbasicsetup}.

We note that the choice of dependent variables plays a significant role in higher dimensions.
In Christodoulou's approach \cite{Christodoulou01}, this is encoded in the fact that energies associated to different multipliers have distinct behavior.
In a similar vein, \cite{LukSpe01} and succeeding works separate vorticity and entropy from other components.
Recently, \cite{ShkVic_2024} made great use of a higher dimensional version of Riemann invariants.
These considerations are quite delicate in higher dimensions, as the geometry of acoustically null (characteristic) hypersurfaces must be taken into account.

\subsection{Maximal globally hyperbolic development}

An MGHD records a maximal region on which \eqref{eq:main} can be classically solved in a causally-closed fashion.
Its boundary therefore tracks some sort of breakdown in either the solution or the causal structure.
In our setting, if $p$ lies on the boundary of an MGHD, either $\mu(p) = 0$ (blowup at $p$), $\mu \to 0$ in the causal past of $p$ (causal breakdown at $p$), or both.
These scenarios correspond to preshocks, Cauchy horizon, and singular boundary, respectively.
For more on this terminology, see Section~\ref{sec:SFstatements}.

To detect these breakdowns, we introduce a quantity $\mu^*(p)$ essentially given by $\inf_{J^-(p)} \mu$, where $J^-(p)$ denotes the causal past of $p$ (defined in Appendix~\ref{sec:appendix1}).
Then the boundary of an MGHD coincides with the zero set of $\mu^*$.
By studying $\mu^*$ in depth, we characterize the boundary of an MGHD in the vicinity of shock formation.

\subsection{Homogeneous expansion}

After showing shock formation under quite general conditions, we assume a stronger form of nondegeneracy (Condition~\ref{cond:genericity}).
After a Galilean transformation, this ensures that $\psi$ resembles an inverse cubic profile $\cub$ solving
\begin{equation*}
  x = -a t \cub + b \cub^3
\end{equation*}
for some $a,b \neq 0$ of the same sign.
We then develop a complete expansion for $\psi$ near shock formation in terms of anisotropic homogeneous functions based on $\cub$.

Our proof relies on the smoothness of $\mu$ and $\psi$ in the eikonal $(\tau, u)$ coordinates described above.
By characterizing the first few terms in the $(\tau, u)$-Taylor expansion of $\mu$, we determine the leading behavior of the (singular) change of variables from $(\tau, u)$ to $(t, x)$.
This allows us to show that $u$ itself resembles $\cub$ near shock formation, and we can recursively develop a complete homogeneous expansion for $u$ as in \cite{CG23}.
Then the expansion for $\psi$ follows from its smoothness in $(t, u)$.

\section{Inviscid structure and geometry}
\label{sec:sfbasicsetup}

In this section, we establish the basic hyperbolic structure of our equation in the inviscid case $\nu = 0$.
The following material is relatively standard (see, for example, \cite{Dafermos} or \cite{BreNotes13}), but we provide derivations for completeness.

We consider the system of equations
\begin{equation}
  \label{eq:psi}
  \partial_t \psi + A(\psi) \partial_x \psi = 0,
\end{equation}
where $\psi$ takes values in an $N$-dimensional vector space $\m{V}$ and $A(\psi)$ is a linear transformation.
Fixing a basis $(e_S)_{S \in \N}$ of $\m{V}$ (and thus an isomorphism to $\R^N$), let $(\psi^S)_{S \in [N]}$ denote the $N$ scalar components of $\psi$.
Adopting Einstein's summation notation, we can then write $\psi = \psi^S e_S$ and
\begin{equation*}
  \partial_t \psi^S + A_T^S(\psi) \partial_x \psi^T = 0.
\end{equation*}
This is the general form of a system of conservation laws, in which $\psi$ is a collection of physical state variables.

\subsection{Spectral notation}
By \ref{hyp:strict-hyp}, $A$ is strictly hyperbolic, having $N$ distinct real eigenvalues.
Numerous physical equations satisfy this condition, including the compressible Euler equations.
Let $\lambda_{(1)} < \ldots < \lambda_{(N)}$ denote these eigenvalues arranged in increasing order.
We write $(I)$ rather than $I$ to emphasize that $\lambda$ is a collection of numbers associated to vectors, not a vector itself.
It is not subject to summation notation.

Let $l^I$ and $r_I$ denote the corresponding families of left and right eigenvectors, so
\begin{equation*}
  l^I A = \lambda_{(I)}l^I, \quad A r_I = \lambda_{(I)} r_I, \And l^J r_I = \delta_I^J.
\end{equation*}
We interpret $l^I$ as a row vector and $r_I$ as a column vector.
We normalize $r_I$ (and thereby $l^I$) by assuming $r_I$ has unit length in the standard inner product on $\R^N$.

Let $R$ denote the matrix changing the standard to the eigenbasis, so $r_I = R_I^S e_S$.
If $(\eta^S)_S$ denotes the basis of row vectors dual to $(e_S)_S$, we likewise define a matrix $L$ by $l^I = L_S^I \eta^S$.
The matrices $R$ and $L$ are inverses, for
\begin{equation*}
  \delta_I^J = l^J r_I = L_S^J R_I^T \eta^S e_T = L_S^J R_I^S.
\end{equation*}
We emphasize that $A$, $\lambda,$ $r$, $l$, $R$, and $L$ all depend on the state $\psi$.
We often elide this dependence, but it is ever-present.

The matrix $A$ acts on $\partial_x \psi$ in \eqref{eq:psi}, so it is natural to decompose $\partial_x \psi$ in terms of the right eigenvectors of $A$: $\partial_x \psi = (\partial_x \psi)^I r_I.$
Using this in \eqref{eq:psi}, we find
\begin{equation*}
  \partial_t \psi + A [(\partial_x \psi)^I r_I] = \partial_t \psi + \lambda_{(I)} (\partial_x \psi)^I r_I = 0.
\end{equation*}
To fully express our equation in terms of the diagonalizing components $(\partial_x \psi)^I$, we apply $\partial_x$:
\begin{equation}
  \label{eq:psix1}
  \begin{aligned}
    0 &= \partial_t [(\partial_x \psi)^I r_I(\psi)] + \partial_x [\lambda_{(I)} (\psi) (\partial_x \psi)^I r_I(\psi)]\\
      &= [\partial_t (\partial_x \psi)^I] r_I + \lambda_{(I)} (\partial_x [\partial_x \psi)^I] r_I + (\partial_x \psi)^I \f {\rd r_I} {\rd \psi^S} \partial_t \psi^S\\
      &\hspace{4cm}+ \left [\f {\rd \lambda_{(I)}} {\rd \psi^S} (\partial_x \psi)^I r_I + \lambda_{(I)} (\partial_x \psi)^I \f {\rd r_I} {\rd \psi^S} \right] (\partial_x \psi)^S\\
      &= [\partial_t (\partial_x \psi)^I] r_I + \lambda_{(I)} [\partial_x (\partial_x \psi)^I] r_I - (\partial_x \psi)^I \f {\rd r_I} {\rd \psi^S} R_J^S \lambda_{(J)} (\partial_x \psi)^J\\
      &\hspace{4cm}+ \left [\f {\rd \lambda_{(I)}} {\rd \psi^S} (\partial_x \psi)^I r_I + \lambda_{(I)} (\partial_x \psi)^I \f {\rd r_I} {\rd \psi^S} \right] (\partial_x \psi)^S.
  \end{aligned}
\end{equation}
Here we have contracted \eqref{eq:psi} with $l^J$ and multiplied by $R_J^S$ to express $\partial_t\psi$ in terms of $\lambda$ and $R$.
These expressions involve derivatives of $r$ and $\lambda$ with respect to $\psi$.
For simplicity, define the 3-tensor $\m{R}$ and the collection of vector $\Lambda$ by
\begin{equation}
  \label{eq:der-eig-rig}
  \f {\rd r_I} {\rd \psi^S} = \mathcal{R}_{I S}^J r_J \And \f {\rd \lambda_{(I)} (\psi)} {\rd \psi^S} = \Lambda_{(I) S}.
\end{equation}
With this notation, we introduce
\begin{equation}
  \label{eq:psixrhs}
  \xi_{J K}^I \coloneqq (\lambda_{(K)} - \lambda_{(J)}) \mathcal{R}_{J S}^I R_K^S - \Lambda_{(I) S} R_J^S \delta_K^I.
\end{equation}
The terms $\mathcal{R}_{J S}^I R_K^S$ and $\Lambda_{(I) S} R_J^S \delta_K^I$ admit nice geometric interpretations.
The first is the Lie bracket in $\m{V}$ between the eigenvectors $r_J$ and $r_K$, and the second is the rate of change of the $I$th eigenvalue in $\m{V}$ in the direction of the $J$th eigenvector.
Collecting the eigencomponents of \eqref{eq:psix1}, we obtain the nonlinear transport system
\begin{equation}
  \label{eq:psix2}
    \partial_t (\partial_x \psi)^I + \lambda_{(I)} \partial_x (\partial_x \psi)^I = \xi_{J K}^I (\partial_x \psi)^J (\partial_x \psi)^K.
\end{equation}
This equation is central to our hyperbolic analysis, as it quantifies the manner in which different components of $\psi$ steepen as a new shock forms.

\subsection{Structural hypotheses}
Our analysis of shock formation relies on two key structural properties, one imposed and one automatic.
The first is the celebrated genuine nonlinearity condition of Lax (see \cite{BreNotes13}).
In our notation, the $I$th wave is genuinely nonlinear if $\xi_{I I}^I = -\Lambda_{(I) S} R_I^S \ne 0$.
As noted above, this represents the (negative) derivative of the eigenvalue $\lambda_{(I)}$ in the direction of its corresponding eigenvector $r_I$.
When $\xi_{II}^I \neq 0$, the $I$th wave feeds itself quadratically in \eqref{eq:psix2}.
This Riccati structure dominates the equation in certain regimes and is well known to lead to blow up in finite time.
We can state \ref{hyp:gen-non} as follows:
\begin{condition}
  \label{cond:gennon}
  There exists an index $I_0 \in [N]$, a constant $c > 0$, and a nonempty open set $U \subset \m{V}$ on which $\absb{\xi_{I_0 I_0}^{I_0}} \geq c$.
\end{condition}
\noindent
We use this genuine nonlinearity to construct solutions that shock in the component $I_0$.

The second key structural property holds for all conservation laws, and follows immediately from \eqref{eq:psixrhs}:
\begin{equation*}
  \xi_{J J}^{I} = 0 \ForAll I \neq J.
\end{equation*}
Hence in \eqref{eq:psix2}, the driver $[(\partial_x \psi)^{I_0}]^2$ only appears in the equation for $(\partial_x \psi)^{I_0}$ itself.
We use this to show that the other components are perturbative as the shock forms.
We use an additional $'$ symbol to index these perturbative waves.
Thus, $I' \in [N] \setminus \{I_0\}$.
When repeated, our Einstein convention will sum $I'$ over all indices not equal to $I_0$.

\subsection{The eikonal function and inverse foliation density}
\label{sec:gquant}

We next define a number of geometric quantities that are intimately tied to shock formation.
We begin with the characteristic vector fields
\begin{equation}
  \label{eq:char-field-L}
  L_{(I)}\coloneqq \partial_t + \lambda_{(I)} \partial_x,
\end{equation}
These of course depend implicitly on the solution $\psi$, through $\lambda_{(I)}$.
We are particularly interested in the shocking vector field $L_{(I_0)}$, to which we associate an eikonal function $u$ satisfying
\begin{equation}
  \label{eq:eikonal}
  L_{(I_0)} u = 0.
\end{equation}
Its level sets are integral curves of $L_{(I_0)}$.
We think of $u$ as tracking wave fronts from initial data, so we set $u(0,x) = x$.
The level sets of the eikonal function $u$ bunch with inverse foliation density
\begin{equation}
  \label{eq:mu}
  \mu \coloneqq \f 1 {\rd_x u}.
\end{equation}

In the following, we will be primarily interested in geometric quantities corresponding to the shocking component $I_0$.
For the sake of brevity, we thus drop the index and write $\lambda$, $r$, and $L$, for $\lambda_{(I_0)}$, $r_{I_0}$, and $L_{(I_0)}$ when confusion will not arise.
We likewise let $\Lambda_S$ denote $\Lambda_{(I_0) S}$.

We are interested in the evolution of the inverse foliation density $\mu$ along the shocking characteristic $L$.
We can first compute
\begin{equation*}
  \begin{aligned}
    L \mu^{-1} = L (\partial_x u) = [L,\partial_x] u &= -\f {\rd \lambda} {\rd \psi^S} (\partial_x \psi)^S \partial_x u = -\Lambda_S R_J^S (\partial_x \psi)^J \mu^{-1}\\
                                                     & = -\Lambda_S R_{I_0}^S (\partial_x \psi)^{I_0} \mu^{-1} - \Lambda_S R_{J'}^S (\partial_x \psi)^{J'} \mu^{-1}.
  \end{aligned}
\end{equation*}
Noting that $L \mu^{-1} = -\mu^{-2} L \mu,$ we obtain
\begin{equation}
  \label{eq:Lmu}
  L \mu = \Lambda_S R_{I_0}^S (\partial_x \psi)^{I_0} \mu + \Lambda_S R_{J'}^S (\partial_x \psi)^{J'} \mu.
\end{equation}
In fact, we can readily check the simpler formula $L \mu = \mu \partial_x \lambda$.
However, we prefer to separate out the shocking component $(\partial_x \psi)^{I_0}$ in \eqref{eq:Lmu}, as it drives the behavior of the system.

\subsection{Eikonal change of coordinates}
\label{sec:eikonal-change}
The eikonal function acts as a ``back to labels'' map: $u(t, x)$ is the initial location of a particle transported by the shocking vector field $L$ to a position $(t, x)$.
We thus make frequent use of $(t, u)$ as a new coordinate system.
To distinguish between partial derivatives in $t$ in which $x$ or $u$ are held constant, we write $t = \tau$ in the new system, so $(t, x) \leftrightarrow (\tau, u)$.
\begin{lemma}
  \label{lem:vec-field-change-of-coord}
  In the coordinate system $(\tau, u)$, $\partial_\tau = L$ and $\partial_u = \mu \partial_x$.
  Conversely, $\partial_t = L - \mu^{-1} \lambda (\mu \partial_x)$ and $\partial_x = \mu^{-1}(\mu \partial_x)$.
\end{lemma}
\begin{proof}
  By the chain rule, $\partial_t = \partial_\tau + \partial_t u \partial_u$ and $\partial_x = \partial_x u \partial_u = \mu^{-1} \partial_u.$
  Hence $\partial_u = \mu \partial_x$ and \eqref{eq:eikonal} yields
  \begin{equation*}
    \partial_\tau = \partial_t - \mu (\partial_t u) \partial_x = \partial_t + \lambda \partial_x = L.
  \end{equation*}
  Rearranging, $\partial_t = L - \mu^{-1} \lambda (\mu \partial_x)$.
\end{proof}
As a simple application, we can differentiate $x$ in the $(\tau, u)$ coordinate system, which informs the behavior of the inverse map expressing $x$ in terms of $(\tau, u)$.
Using Lemma~\ref{lem:vec-field-change-of-coord}, we find
\begin{equation}
  \label{eq:der-x}
  \partial_\tau x = Lx = \lambda \And \partial_u x = \mu.
\end{equation}
We frequently measure regularity relative to the vector fields $L$ and $\mu \partial_x$, which we call \emph{renormalized derivatives}.
We therefore compute commutators between these fields and the nonshocking fields $L_{(I')}$ from \eqref{eq:char-field-L}.
\begin{lemma}
  \label{lem:comm-vec-fields}
  We have $[L,\mu \partial_x] = 0$,
  \begin{equation*}
    [L,\mu L_{(I')}] = (L \mu) L + (\Lambda_{(I') S} - \Lambda_S) R_J^S (L \psi)^J \mu \partial_x,
  \end{equation*}
  and
  \begin{equation*}
    [\mu \partial_x, \mu L_{(I')}] = (\mu \partial_x \mu) L + (\Lambda_{(I') S} - \Lambda_S) R_J^S (\mu \partial_x \psi)^J \mu \partial_x.
  \end{equation*}
\end{lemma}
\begin{proof}
  Coordinate vector fields commute, so $[L, \mu \partial_x] = 0$ follows from Lemma~\ref{lem:vec-field-change-of-coord}.
  For the remaining commutators, we use $\mu L_{(I)} = \mu L + (\lambda_{(I)} - \lambda) \mu \partial_x.$
  It follows that
  \begin{equation*}
    [L,\mu L_{(I)}] = (L \mu) L + L (\lambda_{(I)} - \lambda) \mu \partial_x = (L \mu) L + (\Lambda_{(I) S} - \Lambda_S) R_J^S (L \psi)^J \mu \partial_x.
  \end{equation*}
  Similarly, $[\mu \partial_x,\mu L_{(I)}] = (\mu \partial_x \mu) L + \mu \partial_x (\lambda_{(I)} - \lambda) \mu \partial_x$ yields the third identity.
\end{proof}
We also note for future reference that contracting \eqref{eq:psi} on the left with $l^I$ yields
\begin{equation}
  \label{eq:charvan}
  0 = (\partial_t \psi)^I + \lambda_{(I)} (\partial_x \psi)^I = (L_{(I)} \psi)^I.
\end{equation}

\section{Shock formation statement}
\label{sec:STheorems}

We now focus on the process of shock formation, which is driven by the genuinely nonlinear component $I_0$.
We recall the notation of \eqref{eq:char-field-L} and \eqref{eq:mu}.
\begin{definition}
  We work with the following collection of fundamental unknowns:
  \begin{equation*}
    \Ph \coloneqq \{\psi, (L \psi)^{I'}, (\partial_x \psi)^{I'}, \mu (\partial_x \psi)^{I_0},\mu\}.
  \end{equation*}
  We also write $\Phs \coloneqq \mu (\partial_x \psi)^{I_0}$ and $\Phns^{I'} \coloneqq (\partial_x \psi)^{I'}$.
\end{definition}
The first five quantities in $\Ph$ are derivatives of state variables of the gas, and the last is the geometric quantity $\mu$.
We view these as fundamental because they remain smooth when regularity is measured with respect to $\partial_\tau = L$ and $\partial_u = \mu \partial_x$.
The renormalized derivative $\Phs = \mu (\partial_x \psi)^{I_0}$ plays a particularly important role in our analysis, as $(\partial_x \psi)^{I_0}$ experiences the Riccati blowup at the heart of shock formation.
In \eqref{eq:sx} below, we show that scaling by $\mu$ leads to a less singular equation for $\Phs$.

\subsection{Simple waves}
\label{sub:simple-waves}

We study solutions $\psi$ of \eqref{eq:psi} near so-called simple waves.
These are well-known special solutions of \eqref{eq:psi} that are essentially one-dimensional.
We largely follow the treatment in \cite{Joh90} and elaborate on certain points important for our subsequent analysis.
We direct the reader to \cite{Joh90} and \cite{Dafermos} for further exposition and detail.

A simple wave $\Theta$ associated to $\lambda = \lambda_{I_0}$ is a solution of \eqref{eq:psi} that is constant along the $I_0$ characteristic curves, so $L \Theta = 0$.
Using \eqref{eq:psi}, it follows that
\begin{equation*}
  0 = L \Theta + (A - \lambda I) \partial_x \Theta = (A - \lambda I) \partial_x \Theta.
\end{equation*}
Hence $\partial_x \Theta$ is a scalar multiple of the right eigenvector $r_{I_0}$ corresponding to $\lambda$.
We use this observation to construct such waves.

Motivated by the link between a derivative of $\Theta$ and $r_{I_0}(\Theta)$, consider a solution of the ODE $\Xi' = r_{I_0}(\Xi)$ in $\m{V}$.
Given any solution $\theta$ of the scalar conservation law
\begin{equation}
  \label{eq:simple-eikonal}
  \partial_t \theta + (\lambda \circ \Xi)(\theta) \partial_x \theta = 0,
\end{equation}
the vector-valued composition $\Theta \coloneqq \Xi \circ \theta$ satisfies
\begin{equation*}
  \partial_t \Theta + \lambda(\Theta) \partial_x \Theta = L \Theta = 0
\end{equation*}
and
\begin{equation*}
  \partial_t \Theta + A(\Theta) \partial_x \Theta = \partial_t \Theta + \lambda(\Theta) \partial_x \Theta = 0.
\end{equation*}
That is, $\Theta$ is a simple wave.
We can thus construct a simple wave corresponding to every solution of \eqref{eq:simple-eikonal}.
This is the same eikonal equation as \eqref{eq:eikonal} for $u$, but we typically use different initial data, so we use the new letter $\theta$.

In the following, we recall $\Lambda_S$ from \eqref{eq:der-eig-rig}, the set $U \subset \m{V}$ of genuine nonlinearity from Condition~\ref{cond:gennon}, and the coordinate system $(\tau, u)$ from Section~\ref{sec:eikonal-change}.
\begin{proposition}
  \label{prop:sw}
  For some $a > 0$, let $\Xi \colon [-a, a] \to U$ solve $\Xi' = r_{I_0}(\Xi)$.
  Let $\theta$ solve \eqref{eq:eikonal} with initial data in $\m{C}_\cc^\infty(\R; [-a, a]) \setminus \{0\}$.
  Then the following hold:
\begin{enumerate}[label = \textnormal{(\roman*)}, itemsep = 2pt]
\item $\Theta = \Xi \circ \theta$ is a simple wave solution of \eqref{eq:psi} and $\partial_\tau \Theta = 0.$
  \label{item:simple}
  
\item $\partial_\tau^2 \mu = L^2 \mu = 0$ and $\sup_\tau \min_u L \mu < 0$.
  \label{item:acceleration}
  
\item $\Theta$ forms its first shock at time
  \begin{equation*}
    t_* \coloneqq -[\min_u \partial_\tau \mu(0, u)]^{-1} = -[\min_x (\Lambda_S R_{I_0}^S (\rd_x \Theta)^{I_0}) (0,x)]^{-1} \in \R_+.
    \vspace*{-8pt}
  \end{equation*}
  \label{item:first}
  
\item Of the fundamental unknowns in $\Ph$, only $\psi = \Theta$, $\Phs$, and $\mu$ are nonzero.
  These are smooth in the $(\tau,u)$ coordinate system.
  \label{item:smooth}
\end{enumerate}
\end{proposition}
\begin{proof}
  Lemma~\ref{lem:vec-field-change-of-coord} states that $\partial_\tau = L$, so \ref{item:simple} follows from the preceding discussion.
  
  For \ref{item:acceleration}, we note that $(\rd_x \Theta)^{I'} = 0$, so \eqref{eq:Lmu} and $\Xi' = r_{I_0}(\Xi)$ yield
  \begin{equation}
    \label{eq:Lmu-simple}
    L \mu = \Lambda_S R_{I_0}^S \mu (\rd_x \Theta)^{I_0} = \Lambda_S R_{I_0}^S \mu \partial_x \theta.
  \end{equation}
  We apply $L$ to both sides.
  Because $\Lambda$ and $R$ are functions of $\Theta$ and $L \Theta = 0$, both commute with $L$.
  The same holds for $\mu \partial_x$ (Lemma~\ref{lem:comm-vec-fields}) and $L \theta = 0$, so $L^2 \mu = 0$.
  Hence $L\mu$ is constant in $\tau$, so \eqref{eq:Lmu-simple} yields
  \begin{equation*}
    \sup_\tau \min_u L \mu = \min_u L \mu|_{\tau = 0} = \min_u \Lambda_S R_{I_0}^S \mu \partial_x \theta|_{\tau = 0}.
  \end{equation*}
  Now $\Lambda_S R_{I_0}^S \neq 0$ by genuine nonlinearity, $\mu|_{\tau = 0} = 1$, and $\partial_x \theta|_{\tau = 0}$ attains both positive and negative values because $\theta|_{\tau = 0}$ is compactly supported and nonzero.
  It follows that $\min_u L \mu|_{\tau = 0} < 0$.
  
  We have just shown that at the initial time $\tau = 0$, $\mu = 1$ and $\partial_\tau \mu = \Lambda_S R_{I_0}^S \partial_x \theta$.
  Integrating $\partial_\tau^2 \mu = 0$ twice in $\tau$, we conclude that
  \begin{equation}
    \label{eq:mu-lin}
    \mu(\tau, u) = 1 + (\Lambda_S R_{I_0}^S \partial_x \theta)(0, u) \tau.
  \end{equation}
  Choose $u_* \in \R$ minimizing $\Lambda_S R_{I_0}^S \partial_x \theta|_{\tau = 0}$ and define $t_*$ as in \ref{item:first}.
  Then $\mu(t_*, u_*) = 0$ but $\inf_u \mu(\tau, u) > 0$ for all $\tau < t_*$.
  So the change of coordinates $(t, x) \leftrightarrow (\tau, u)$ breaks down precisely at time $t_*$.
  Moreover, $L \mu \partial_x \theta = 0$ and the definition of $u_*$ ensures that $\mu \partial_x \theta(0, u_*) \neq 0$.
  Hence $\abss{(\partial_x \Theta)^{I_0}} \asymp \mu^{-1} \to \infty$ as $\tau \nearrow t_*$.
  Therefore \ref{item:first} holds: $\Theta$ is smooth before time $t_*$ but develops a shock (infinite slope) at $t_*$.

  Finally, for \ref{item:smooth}, we have $L\Theta = 0$ and $(\partial_x \Theta)^{I'} = 0$.
  Differentiating the nontrivial quantities in $\tau$, we find $\partial_\tau \theta = 0$, $\partial_\tau \Phs = \partial_\tau (\mu \partial_x \theta) = \partial_\tau \partial_u \theta = 0,$ and $\partial_\tau \mu = \Lambda_S R_{I_0}^S \Phs$ while $\partial_\tau^2 \mu = 0$.
  The quantities $\theta,\Phs,$ and $\mu$ are initially smooth, and these identities show that this smoothness is propagated.
\end{proof}
Due to the freedom in $\theta$, this proposition yields a large class of shocking simple waves.
For simplicity, in the remainder of the paper we fix a $\theta$ satisfying $\min (\Lambda_S R_{I_0}^S \partial_x \theta)|_{\tau = 0} = -1$, so that the first shock develops at time $t_* = 1$.
We observe that the resulting simple wave $\Theta = \Xi \circ \theta$ takes values in a compact subset of $\m{V}$.
We allow all free constants to implicitly depend on $A$ and its derivatives in this compact subset.

We wish to show that shock formation is stable near simple waves.
However, we are not able to do so in complete generality when the shocking characteristic $I_0$ has intermediate speed.
We therefore impose an extremely weak nondegeneracy condition on $\Theta$.
\begin{condition}
  \label{cond:L1Th}
  A simple wave $\Theta$ is \emph{mildly nondegenerate} with parameters $\eta, \delta_1 > 0$ and $\delta_2 \in (0, 1/10]$ if one of the following holds:
  \begin{enumerate}[label = \textnormal{(\roman*)}, itemsep = 2pt]
  \item $\Theta$ corresponds to an extremal characteristic: $I_0 = 1$ or $N$.
    
  \item There exist $u_1 < u_2$ with $u_2 - u_1 \leq 2\eta$ such that for all $0 \leq \tau \leq t_* + \delta_1$ and $u \in [u_1,u_2]^\cc$, $\mu(\tau, u) \geq \delta_2$.
    Additionally, $L \mu \leq -\frac{3}{4} \inf L \mu$ in $[u_1 - \delta_1, u_2 + \delta_1]$.
  \end{enumerate}
\end{condition}
\noindent
This condition prevents uncontrolled growth in the linearized equations around the background wave---see Section~\ref{sec:badlin} for details.
We emphasize that there is no need for such a condition when $I_0$ is extremal.
When $I_0$ has intermediate speed, we will require simple waves with $\eta \ll 1$.
To construct such waves, take $\eta > 0$ arbitrarily small and fix $u_1 < u_2 \leq u_1 + \eta$.
Then choose data for $\theta$ such that
\begin{equation*}
  L \mu|_{\tau = 0} = \Lambda_S R_{I_0}^S \partial_x \theta|_{\tau = 0} > -1 \quad \text{in } (u_1, u_2)^\cc
\end{equation*}
but $\min_{(u_1,u_2)} L \mu|_{\tau = 0} = -1$.
Because $\supp \theta \setminus (u_1, u_2)$ is compact, \eqref{eq:mu-lin} implies the existence of $\delta_1 > 0$ and $\delta_3 \in (0, 1/10]$ such that $\mu(\tau,u) \ge \delta_3$ for all $0 \le \tau \le 1 + \delta_1$ and $u \in (u_1, u_2)^\cc$.
This wave satisfies $t_* = 1$ and Condition~\ref{cond:L1Th}.

Examining this construction, we note that Condition~\ref{cond:L1Th} allows for an infinitely degenerate configuration on an interval, as long as the interval has small length.
An example is data for Burgers where the minimum of the spatial derivative of the initial data is attained on an entire (small) interval.
Thus, in particular, every finitely degenerate situation ought to satisfy a condition like this as we approach the time of singularity formation.

\subsection{Shock formation formulation}
\label{sec:SFstatements}

To frame our main theorem on shock formation, we use the notion of a maximal globally hyperbolic development (MGHD) for a hyperbolic equation.
Informally, an MGHD is a solution over a spacetime region that cannot be extended.
Due to finite speed of propagation,  we can define an MGHD that extends past time $t_*$ away from the first singularity.
For the reader's convenience, we describe MGHDs in more precise terms in Appendix~\ref{sec:appendix1}.
\begin{definition}
  We let $(\mathcal{M}, \psi)$ denote a maximal globally hyperbolic development, where $\mathcal{M} \subset \R^{1 + 1}$ is the domain of definition of the solution $\psi$.
  Given $t > 0$, we let $\mathcal{M}_t$ denote $([0,t] \times \R) \cap \mathcal{M}$ and $\calB_t$ the future boundary of $\calM_t$.
\end{definition}
\noindent
In \cite{ShkVic_2024}, $\mathcal{M}_t$ is termed a maximal globally hyperbolic development ``in a box.''
It allows one to study the boundary of the MGHD near the first time singularity while suppressing more global issues.

We are interested in this future boundary $\m{B}_t$ because it tracks the most severe effects of the singularity: those that prevent the solution from being extended.
Broadly speaking, boundary points of an MGHD satisfy at least one of two conditions: they are singular (a crucial derivative of the solution blows up), or they contain a singular point in their causal past (defined in Appendix~\ref{sec:appendix1}).
We decompose the future boundary $\m{B}_t$ of $\m{M}_t$ according to which conditions are satisfied.

Intuitively, the ``preshock set'' $\m{B}_t^{\text{pre}}$ consists of singularities with regular causal past.
Conversely, points in the ``Cauchy horizon'' $\m{B}_t^{\text{Cau}}$ are regular but have singular causal past.
The ``singular set'' $\m{B}_t^{\text{sing}}$ captures points that are both singular and contain singularities in their causal past.
Finally, the ``extensible set'' $\m{B}_t^{\text{ext}}$ avoids all singularities and has time coordinate $t$, so it lies on the temporal horizon of the box.
With this terminology $\m{B}_t$ is the disjoint union $\m{B}_t = \m{B}_t^{\text{pre}} \sqcup \m{B}_t^{\text{Cau}} \sqcup \m{B}_t^{\text{sing}} \sqcup \m{B}_t^{\text{ext}}$.

In greater detail, the preshock set $\m{B}_t^{\text{pre}}$ represents the birth of a new singularity.
Generically, it consists of isolated points, known as preshocks or creases.
Assuming $t$ is only slightly larger than the time coordinate of the earliest preshock, we are free to assume that there is a single preshock.
Much of our analysis revolves around the behavior of the solution near $\m{B}_t^{\text{pre}}$.
The union $\m{B}_t^{\text{Cau}} \sqcup \m{B}_t^{\text{sing}}$ generically consists of two curves emanating from the preshock.
If the shocking characteristic $I_0$ has intermediate speed, both curves lie in the Cauchy horizon $\m{B}_t^{\text{Cau}}$ and $\m{B}_t^{\text{sing}} = \emptyset$.
In the extremal setting, one curve forms the Cauchy horizon and the other constitutes the so-called singular manifold $\m{B}_t^{\text{sing}}$.
Finally, as its name suggests, the solution is perfectly well-behaved near the extensible set $\m{B}_t^{\text{ext}}$ and could be continued in a neighborhood.
This portion of the boundary $\m{B}_t$ is merely an artifact of our focus on times before $t$.
The above descriptions are meant to give an intuitive picture of the various components of $\m{B}_t$.
For precise definitions, see Section~\ref{sec:MGHD}.

We place one more  condition on the background solution.
\begin{condition}
  \label{cond:MGHDstab}
  We say $\Theta$ satisfies the \emph{graphical condition} if there exists a smooth vector field $\rd_t + F_\Theta (t,x) \rd_x$ such that $\lambda_{(1)} \circ \Theta < F_\Theta < \lambda_{(N)} \circ \Theta$ uniformly.
\end{condition}
This holds locally provided the singularities in $\Theta$ are not too degenerate.
For example, it suffices for $\Theta$ to take values in a small ball in state space.
In particular, Condition~\ref{cond:MGHDstab} holds for generic shocks.

For more general waves $\Theta$, the algorithm we use to construct the MGHD will not necessarily produce a unique output without such a condition.
Indeed in its absence, it seems possible to produce a double-sheeted cover of part of the $(t,x)$ plane.
For more details, we refer the reader to~\cite{EpReSb19}, which first introduced graphical conditions like Condition~\ref{cond:MGHDstab}.
Double-sheeted covers can also occur in scalar conservation laws such as the Burgers equation~\cite{AbbSpe23}.
For this reason, we require $N \geq 2$ in Theorem~\ref{thm:SF}, so that we have a genuine system.
These issues are intimately tied to the uniqueness of MGHDs, which we do not address in this work.

We can now state our main theorem on shock formation: the shocks of simple waves are stable and smooth measured in $(\tau, u)$ coordinates.
We recall that the following constants implicitly depend on the compact region in the state space $\m{V}$ where our waves take values.
\begin{theorem}
  \label{thm:SF}
  Fix $N \ge 2$.
  There exists $\eta > 0$ such that the following holds.
  Let $\Theta$ be a simple wave solution of \eqref{eq:psi} constructed in Proposition~\ref{prop:sw} that first forms a shock at time $1$ and satisfies Conditions~\ref{cond:L1Th} and \ref{cond:MGHDstab} with parameters $(\eta, \delta_1,\delta_2)$ for some $\delta_1,\delta_2 > 0$.
  Then there exist $\delta,\eps_0 > 0$ such that for any $\eps \in (0, \eps_0]$ and $v \in B_{\m{C}_x^1}(\eps)$, the solution $\psi$ of \eqref{eq:psi} with $\psi(0, \anon) = \Theta(0, \anon) + v$ forms a shock in finite time.
  Moreover:
  \begin{enumerate}[label = \textnormal{(\roman*)}, itemsep = 2pt]
  \item
    \label{item:shock-time}
    There exists an earliest time $t_* = 1 + \m{O}(\eps)$ for which $\smash{\liminf\limits_{t \nearrow t_*}} \norm{\nab \psi}_{L_x^\infty} (t) = \infty$ while $\psi$ remains bounded.

  \item
    \label{item:MGHD}
    We construct $(\calM_{t_* + \delta}, \psi)$ and characterize the components of the disjoint union $\m{B}_{t_* + \delta} = \m{B}_{t_* + \delta}^{\textnormal{pre}} \sqcup \m{B}_{t_* + \delta}^{\textnormal{Cau}} \sqcup \m{B}_{t_* + \delta}^{\textnormal{sing}} \sqcup \m{B}_{t_* + \delta}^{\textnormal{ext}}$.
    The set $\calB_{t_* + \delta}$ is uniformly Lipschitz.

  \item
    \label{item:fund-smooth}
    Suppose $\psi(0, \anon) \in \m{C}_x^{k+1}$ for some $k \geq 0.$
    Then
    \begin{equation*}
      \norm{\Ph}_{\m{C}_{\tau, u}^k(\m{M}_{t_* + \delta})} \leq M \norm{\psi(0, \anon)}_{\m{C}_x^{k+1}}
    \end{equation*}
    for some constant $M(\Theta,A, \eta, \delta_1, \delta_2, \delta, k) > 0$.
  \end{enumerate}
\end{theorem}
We state this result for $N \geq 2$ so we can push the analysis slightly after the time $t_*$ of shock formation.
If one is only interested in $t \leq t_*$, we can apply this theorem to scalar conservation laws ($N=1$) by adding a decoupled linear transport equation to form a system.
This auxiliary equation is irrelevant when $t \leq t_*$, but it will alter the nature of the MGHD after $t_*$.
With this trick, our subsequent viscous analysis applies to scalar laws as well.

As a consequence of \ref{item:fund-smooth}, the solution $\psi$ is as smooth as data allows in the $(\tau,u)$ coordinate system.
In particular, it has a Taylor expansion to suitable order in these coordinates.
In Section~\ref{sec:homogeneous}, we thereby show that $\psi$ admits an expansion in homogeneous functions founded on a universal cubic cusp.
This proves central to our subsequent study of the inviscid limit.

By \ref{item:MGHD}, the boundary $\m{B}_{t_* + \delta}$ of the MGHD in a box is uniformly Lipschitz.
In the generic case of a single preshock, the curvature of $\calB_{t_* + \delta}$ is in fact a Radon measure (see, for example, \cite{AbbSpe23}).
Indeed, $\m{B}_{t_* + \delta}$ then resembles two smooth curves emanating from a single point, with a point-mass of curvature at the vertex determined by the (positive) angle between the two curves.
However, at the level of generality considered in Theorem~\ref{thm:SF}, the regularity of $\calB_{t_* + \delta}$ in \ref{item:MGHD} is sharp in the sense that there exist very degenerate shocks for which the curvature of $\calB_{t_* + \delta}$ is not a Radon measure.
\begin{proposition}
  \label{prop:Bcurvature}
  There exists an MGHD $(\calM,\psi)$ such that $\psi$ has smooth initial data and first forms a shock at time $1$, but for all $\delta > 0$, the curvature of $\calB_{1 + \delta}$ is not a Radon measure.
\end{proposition}
\noindent
This phenomenon can arise in very simple equations.
We exhibit it in the Burgers equation paired with decoupled linear transport.

\section{Renormalized equations}
\label{sec:reneq}

Our approach to shock formation relies on controlling the fundamental unknowns $\Ph$ in the $(\tau, u)$ coordinates.
Recall that the corresponding partial derivatives are $\partial_\tau = L$ and $\partial_u = \mu \partial_x$.
Crucially, the potentially singular spatial derivative is ameliorated by the inverse foliation density $\mu$, which vanishes at shock formation.
In this section, we show that this ``renormalization'' results in more favorable evolution equations.

\subsection{Equations for fundamental unknowns}
Due to our interest in $L$ and $\mu \partial_x$, we find equations for $(L \psi)^I$, $(\mu \rd_x \psi)^{I_0} = \Phi$, and $(\rd_x \psi)^{I'} = \phi^{I'}$.

Recalling \eqref{eq:charvan}, we observe that $(L \psi)^{I_0} = 0.$
Writing $L_{(I')} = L + (\lambda_{(I')} - \lambda) \rd_x$ and using \eqref{eq:charvan}, we also see that
\begin{equation}
  \label{eq:Lpsicontrol}
  (L \psi)^{I'} = (\lambda - \lambda_{(I')}) (\rd_x \psi)^{I'} = (\lambda - \lambda_{(I')}) \phi^{I'}.
\end{equation}
Thus to control $(L \psi)^{I'}$, it suffices to control $\phi^{I'}$.
Similarly, we can commute \eqref{eq:Lpsicontrol} with $L$ and $\mu \partial_x$ to control higher derivatives of $(L \psi)^I$ in terms of those of $\psi$ and $\Phns^{I'}$.
With this in mind, we can implicitly remove $(L \psi)^I$ from consideration in $\Ph$ and control it afterward using \eqref{eq:Lpsicontrol}.

Turning to $\Phi = \mu (\partial_x \psi)^{I_0}$, we take $I = I_0$ in \eqref{eq:psix2} and multiply by $\mu$:
\begin{equation*}
  \mu \partial_t (\partial_x \psi)^{I_0} + \mu \lambda \partial_x (\partial_x \psi)^{I_0} = \mu L (\partial_x \psi)^{I_0} = \mu \xi_{J K}^{I_0} (\partial_x \psi)^J (\partial_x \psi)^K.
\end{equation*}
Commuting $\mu$ inside the equation gives
\begin{equation}
  \label{eq:ominous}
  \begin{aligned}
    L \Phs &= (L \mu) (\partial_x \psi)^{I_0} + \mu \xi_{I_0 I_0}^{I_0} (\partial_x \psi)^{I_0} (\partial_x \psi)^{I_0} + \mu \xi_{J' K'}^{I_0} (\partial_x \psi)^{J'} (\partial_x \psi)^{K'}\\
           &\quad + \mu (\xi_{I_0 J'}^{I_0} + \xi_{J' I_0}^{I_0}) (\rd_x \psi)^{I_0} (\rd_x \psi)^{J'} 
    \\ &= \mu^{-1} \xi_{I_0 I_0}^{I_0} \Phs^2 + \mu \xi_{J' K'}^{I_0} \Phns^{J'} \Phns^{K'} + (\xi_{I_0 J'}^{I_0} + \xi_{J' I_0}^{I_0}) \Phs \Phns^{J'} + \mu^{-1} (L \mu) \Phs.
  \end{aligned}
\end{equation}
The first term on the right suggests Riccati blowup, but in fact $\mu^{-1} (L \mu) \Phs$ contributes an equal but opposite term, leading to a more favorable equation.
Recall from \eqref{eq:psixrhs} that $\xi_{I_0 I_0}^{I_0} = -\Lambda_S R_{I_0}^S$.
Moreover, we can write \eqref{eq:Lmu} as $L \mu = \Lambda_S R_{I_0}^S \Phs + \Lambda_S R_{J'}^S \mu \Phns^{J'}.$
Combining these observations in \eqref{eq:ominous}, we cancel the dangerous term and obtain
\begin{equation}
  \label{eq:sx}
  L \Phs = (\xi_{I_0 J'}^{I_0} + \xi_{J' I_0}^{I_0} + \Lambda_S R_{J'}^S) \Phs \Phns^{J'} + \mu \xi_{J' K'}^{I_0} \Phns^{J'} \Phns^{K'}.
\end{equation}
Finally, equations for $\phi^{I'} = (\partial_x \psi)^{I'}$ follow from simply multiplying \eqref{eq:psix2} by $\mu$:
\begin{equation} \label{eq:nsx}
  \begin{aligned}
    \mu L \Phns^{I'} + (\lambda_{(I')} - \lambda) \mu \partial_x \Phns^{I'} &= \mu \xi_{J K}^{I'} (\partial_x \psi)^J (\partial_x \psi)^K\\
                                                                            &= (\xi_{I_0 J'}^{I'} + \xi_{J' I_0}^{I'}) \Phs \Phns^{J'} + \mu \xi_{J' K'}^{I'} \Phns^{J'} \Phns^{K'}.
  \end{aligned}
\end{equation}
In conjunction with \eqref{eq:psi} and \eqref{eq:Lmu}, this completes our derivation of equations for the fundamental unknowns $\Ph$.

\subsection{Bad linear terms}
\label{sec:badlin}

We are now faced with the primary technical hurdle in our hyperbolic analysis: the system \eqref{eq:nsx} over $I'$ is not amenable to Gr\"onwall's inequality when $I_0$ has intermediate speed.
To see this, we write \eqref{eq:nsx} in the $(\tau,u)$ coordinate system schematically as
\begin{equation*}
  (\lambda_{(I')} - \lambda) \rd_u \Phns^{I'} = \m{O}(1) \Phns^{J'} + \m{O}(\mu).
\end{equation*}
This is a system of ODEs in the eikonal coordinate $u$.
If $I_0$ has intermediate speed, $\lambda_{(I')} - \lambda$ assumes both positive and negative values as $I'$ varies.
Hence there are both left-moving and right-moving waves with no preferred direction for Gr\"onwall.
In short, the system resembles an ODE \emph{boundary} value problem rather than initial value problem.

These linear terms pose a potential obstruction to shock stability.
Indeed, consider an idealized scenario in which the shock is infinitely degenerate on the whole line, so $\mu = 1 - \tau$.
Replacing $\Phi$ by $\partial_u \Theta$ and neglecting the final term in \eqref{eq:nsx}, the equation becomes
\begin{equation*}
  (1 - \tau) \partial_\tau \phi + \mathbf{L}\phi \approx 0,
\end{equation*}
where $\mathbf{L}_{J'}^{I'} = (\lambda_{(I')} - \lambda) \delta_{J'}^{I'} \partial_u -  (\xi_{I_0 J'}^{I'} + \xi_{J' I_0}^{I'}) \partial_u \Theta$.
Suppose $\psi$ is an eigenfunction of $\mathbf{L}$ with eigenvalue $\gamma$.
Then if $\phi = a(\tau) \psi$, the amplitude $a$ solves $(1 - \tau) \dot a + \gamma a = 0,$ so $a(\tau) = C (1 - \tau)^\gamma$.
In particular, if $\gamma < 0$, perturbations of $\Theta$ in direction $\psi$ diverge at the time of shock, an instability.

We emphasize that this issue only occurs when $I_0$ has intermediate speed.
If $I_0$ is extremal, $\lambda_{(I')} - \lambda$ all share a sign, so \eqref{eq:nsx} \emph{does} have a preferred direction in which we can apply Gr\"onwall.
This is why we place no restriction on extremal shocks in Condition~\ref{cond:L1Th}.
Moreover, the schematic discussion above involves a shock degenerating on the whole line.
As we shall see, restricting degeneracy to a small region suffices to control this kind of behavior.
Thus Condition~\ref{cond:L1Th} permits infinitely degenerate shocks, provided the degeneracy is confined to a small interval.

\subsection{Perturbative unknowns}

We now introduce the perturbative unknowns, denoted by $\Php$.
Quantities in $\Php$ measure the difference between the ``background'' $\Theta$ and the nearby solution $\psi$.
They are thus small in $\eps$, while the original set $\Ph$ contains quantities that are order one.

Of course, we wish to compare $\psi$ itself with $\Theta$, but we shouldn't directly study $\psi(t, x) - \Theta(t, x)$.
Indeed, these functions generally become singular at slightly different points, so their difference need not be small in, say, $\m{C}^1$.
To circumvent this, one could modulate coordinates to align their singularities, but we take a slightly different approach.
Instead, we consider the difference $\psi(\tau, u) - \Theta(\tau, u)$ in the eikonal coordinates.
This amounts to identifying the null gauges adapted to $\psi$ and $\Theta$ respectively, a common procedure in the study of general relativity.
In the same manner, we compare the fundamental unknowns $\Ph$ associated to $\psi$ and $\Theta$ in the $(\tau, u)$ coordinate system.
We note that the simple wave $\Theta$ is automatically globally defined in $(\tau, u)$.

To carry out this comparison, we linearize around $\Theta$ in the $(\tau, u)$ coordinate system.
To ease notation, we introduce notation tracking background and difference quantities.
\begin{definition}
  \label{def:pert-func}
  We use a hat to label quantities based on $\Theta$ rather than $\psi$.
  The fundamental unknowns for $\Theta$ are $\PhT$ and, for example,
  \begin{equation*}
    \psiT(\tau,u) \coloneqq \Theta(\tau,u), \quad \AT(\tau,u) \coloneqq (A \circ \Theta)(\tau,u), \And \rT_I (\tau,u) \coloneqq (r_I \circ \Theta)(\tau,u).
  \end{equation*}  
  To define perturbative unknowns, we subtract quantities for $\psi$ by corresponding quantities for $\Theta$ in $(\tau, u)$ coordinates.
  We denote these with a bar: $\psip \coloneqq \psi - \Theta$, $\Phsp \coloneqq \Phs - \PhsT$, $\mup \coloneqq \mu - \muT$, etc.
  The fundamental perturbative unknowns are $\bar{\mathcal{P}}.$
\end{definition}
We now derive equations for the quantities in $\bar{\Ph}$ in the $(\tau, u)$ coordinates.
\begin{lemma}
  \label{lem:perteq}
  Let $\XiT_J^I \coloneqq (\xiT_{I_0 J}^I + \xiT_{J I_0}^I) \PhsT$, which depends only on $\Theta$.
  Then:
  \begin{enumerate}[label = \textup{(\roman*)}]
  \item
    \label{item:psip}
    For $\psip = \psi - \psiT$, $\partial_\tau \bar\psi = \partial_\tau \psi$.
    
  \item
    \label{item:Phsp}
    \vspace*{2pt}
    For $\Phsp = \Phs - \PhsT$,
    \begin{equation}
      \label{eq:Phsp}
      \begin{aligned}
        \partial_\tau \Phsp = (\XiT_{J'}^{I_0} + &\LaT_S \RT_{J'}^S) \Phns^{J'} + (\xip_{I_0 J'}^{I_0} + \xip_{J' I_0}^{I_0}) \PhsT \Phns^{J'} + (\xi_{I_0 J'}^{I_0} + \xi_{J' I_0}^{I_0}) \Phsp \Phns^{J'}\\
                                                 &+ \mu \xi_{J' K'}^{I_0} \Phns^{J'} \Phns^{K'}+ \LaT_S \Rp_{J'}^S \PhsT \Phns^{J'} + \Lap_S R_{J'}^S \PhsT \Phns^{J'} + \Lambda_S R_{J'}^S \Phsp \Phns^{J'}.
      \end{aligned}
    \end{equation}

  \item
    \label{item:Phns}
    For $\bar\Phns^{I'} = \Phns^{I'}$,
    \begin{equation}
      \label{eq:Phns}
      \begin{aligned}
        \mu \rd_\tau \Phns^{I'} + (\lambda_{(I')} - \lambda) \rd_u \Phns^{I'} = \XiT_{J'}^{I'} \Phns^{J'} &+ (\xip_{I_0 J'}^{I'} + \xip_{I_0 J'}^{I'}) \PhsT \Phns^{J'} \\
        &+(\xi_{I_0 J'}^{I'} + \xi_{J' I_0}^{I'}) \Phsp \Phns^{J'}+ \mu \xi_{J' K'}^{I'} \Phns^{J'} \Phns^{K'}.
      \end{aligned}
    \end{equation}
    
  \item
    \label{item:mup}
    For $\mup = \mu - \muT$,
    \begin{equation}
      \label{eq:mup}
        \partial_\tau \mup = \La_S R_{I_0}^S \Phsp + \Lap_S R_{I_0}^S \PhsT + \LaT_S \Rp_{I_0}^S \PhsT + \mu \Lambda_S R_{J'}^S \Phns^{J'}
    \end{equation}
    and
    \begin{equation}
      \label{eq:L2mu}
        \partial_\tau^2 \mu = \partial_\tau^2 \mup + \partial_\tau^2 \muT = \partial_\tau^2 \mup.
      \end{equation}
  \end{enumerate}
\end{lemma}
\begin{proof}
  Because $L$ is a simple wave, $\partial_\tau \h \psi = L \Theta = 0$, and \ref{item:psip} follows.

  For \ref{item:Phsp}, the same implies that $\partial_\tau \h\Phs = 0$,
  Hence \eqref{eq:sx} becomes
  \begin{equation}
    \label{eq:sx-expand}
    \partial_\tau \Phsp = (\xi_{I_0 J'}^{I_0} + \xi_{J' I_0}^{I_0}) \Phs \Phns^{J'} + \mu \xi_{J' K'}^{I_0} \Phns^{J'} \Phns^{K'} + \Lambda_S R_{J'}^S \Phs \Phns^{J'}.
  \end{equation}
  Writing $\Phs = \PhsT + \Phsp$ and $\xi = \xiT + \xip$, we can expand the first factor on the right side as
  \begin{equation*}
    (\xi_{I_0 J'}^{I_0} + \xi_{J' I_0}^{I_0}) \Phs = \XiT_{J'}^{I'} + (\xip_{I_0 J'}^{I'} + \xip_{I_0 J'}^{I'}) \PhsT + (\xi_{I_0 J'}^{I'} + \xi_{J' I_0}^{I'}) \Phsp
  \end{equation*}
  with $\XiT_J^I = (\xiT_{I_0 J}^I + \xiT_{J I_0}^I) \PhsT$.
  Similarly, $\Lambda_S R_{J'}^S = \LaT_S \RT_{J'}^S + \LaT_S \Rp_{J'}^S + \Lap_S R_{J'}^S.$
  With these observations, \eqref{eq:sx-expand} becomes \eqref{eq:Phsp}.

  Because $\Theta$ varies in $r_{I_0}$ alone, $\h{\Phns} = 0$ and $\bar{\Phns} = \Phns$.
  Expanding $(\xi_{I_0 J'}^{I_0} + \xi_{J' I_0}^{I_0}) \Phs$ as above, \eqref{eq:nsx} becomes \eqref{eq:Phns} and we obtain \ref{item:Phns}.

  Finally, for \ref{item:mup}, \eqref{eq:Lmu} yields
  \begin{align*}
    \partial_\tau \mup = \partial_\tau \mu - \partial_\tau \muT &= \Lambda_S R_{I_0}^S \Phs + \mu \Lambda_S R_{J'}^S \Phns^{J'} - \LaT_S \RT_{I_0}^S \PhsT\\
                                                                &= (\LaT_S + \Lap_S) (\RT_{I_0}^S + \Rp_{I_0}^S) (\PhsT + \Phsp) + \mu \Lambda_S R_{J'}^S \Phns^{J'} - \LaT_S \RT_{I_0}^S \PhsT.
  \end{align*}
  Gathering terms, this implies \eqref{eq:mup}.
  Then \eqref{eq:L2mu} follows from Proposition~\ref{prop:sw}.
\end{proof}

\subsection{Structural notation}

In Lemma~\ref{lem:perteq}, we collected the equations for the perturbative quantities $\Php$.
We now introduce structural notation for the kinds of expressions that appear upon commutation of these equations.
In the following, let $\Gamma$ denote one of the commuting vector fields $\partial_\tau$ or $\partial_u$.
Given a multi-index $\al = (\al_1,\al_2)$, let $\Gamma^\alpha = \partial_\tau^{\al_1} \partial_u^{\al_2}$.
\begin{definition}
  Given $k \in \N_0$ and $\m{S} \subset \Php$, let $\m{N}_k[\m{S}]$ denote an expression that is a sum of at most $(N + 10 + k)^{N + 10 + k}$ terms of the form $H \Gamma^\al \m{S}$, where $H$ is a smooth function of $\PhT$ and $\Php$, $\abs{\al} \leq k$, and $\Gamma^\al \m{S}$ denotes $\Gamma^\al$ applied to one of the quantities in $\m{S}$.
  We omit $\m{S}$ when $\m{S} = \Php$, so $\m{N}_k = \m{N}_{k}[\Php]$.
  We let $\m{Q}_k$ denote a similar expression whose terms have the form $H(\Gamma^\al \Php)(\Gamma^\beta \Php)$ with $\abs{\al} + \abs{\beta} \leq k$.
\end{definition}
The bound $(N + 10 + k)^{N + 10 + k}$ on the number of terms accounts for various combinatorial possibilities.
We only mention it to indicate that the dependence of implied constants on $N$ and $k$ can (in principle) be tracked explicitly.

With this notation, we can use Lemma~\ref{lem:perteq} and induction to show:
\begin{proposition}
  \label{prop:schematiceqs}
  For all $\al \in \N_0^2$,
  \begin{align*}
    \mu \partial_\tau (\Gamma^\alpha \Phns^{I'}) + (\lambda_{(I')} - \lambda) \rd_u (\Gamma^\alpha \Phns^{I'}) &= \h{\Xi}_{J'}^{I'} \phi^{J'} + \calN_{|\alpha| - 1} + \calQ_{|\alpha|},\\
    \partial_\tau (\Gamma^\alpha \Phsp) &= \m{N}_{\abs{\al}}[\Phns^{I'}] + \calN_{|\alpha| - 1} + \calQ_{|\alpha|},\\
    \partial_\tau (\Gamma^\alpha \psip) &= \calN_{|\alpha|}[\Phns^{I'},\Phsp] + \calN_{|\alpha| - 1},\\
    \partial_\tau (\Gamma^\alpha \mup) &= \calN_{|\alpha|}[\Phns^{I'},\Phsp,\psip] + \calQ_{|\alpha|}.
  \end{align*}
\end{proposition}

\section{Hyperbolic estimates} \label{sec:HyperbolicEsts}

We now prove our fundamental hyperbolic estimates, which come in two flavors depending on the extremality of $I_0$.

We formulate these bounds for a rather general hyperbolic system governing $N$ unknown functions $f^1, \dots, f^N$.
Let $V_{(1)}, \dots, V_{(N)}$ denote $N$ (not necessarily distinct) Lipschitz vector fields and assume the $t$-component of each has a uniform lower bound.
Let $\m{D}$ be a connected globally hyperbolic domain for these vector fields with Cauchy hypersurface $\Sigma$.
Let $\gamma_{I,p}$ denote the unique future-directed affinely parametrized integral curve of $V_{(I)}$ starting at $\Sigma$ and ending at $p$.
That is, $\dot{\gamma}_{I,p} = V_{(I)} \circ \gamma_{I,p}$, $\gamma_{I,p}(0) \in \Sigma$, and $\gamma_{I,p}\big(s_I(p)\big) = p$ for a unique time $s_I(p) > 0$.
Let $\m{T}_J^I$ and $\Pi^I$ denote a continuous bounded matrix and vector, respectively.
Motivated by Lemma~\ref{lem:perteq}, we consider the following system of equations:
\begin{equation}
  \label{eq:lineqs}
  V_{(I)} (f^{I}) = \calT_J^I f^J + \Pi^I
\end{equation}
Because the $V_{(I)}$ have uniformly positive time component and the equation is linear, the existence of $f^I$ in $\m{D}$ follows from standard local existence theory.

Our first estimate applies to extremal shocks.
\begin{proposition}
  \label{prop:r1est}
  Suppose $v$ is a Lipschitz function on $\m{D}$ for which $V_{(I)} v$ has the same sign for all $I$ and $M \coloneqq \max_I \sup_{\m{D}} |V_{(I)} v|^{-1} < \infty$.
  Let
  \begin{equation*}
    D_0 \coloneqq \max_I \sup_{\m{D}} |f^I|_{\Sigma}|, \quad D_1 \coloneqq \max_I \sum_J \sup_{\m{D}} |\calT_J^I|, \And D_2 \coloneqq \max_I \sup_{\m{D}} |\Pi^I|.
  \end{equation*}
  Then
  \begin{equation}
    \label{eq:r1est}
    \max_I \sup_{\m{D}} |f^{(I)}| \le (D_0 + M D_2 \op{osc}v) \exp(M D_1 \op{osc} v).
  \end{equation}
\end{proposition}
\begin{proof}
  Perhaps replacing $v$ by $-v$, we are free to assume that $V_{(I)} v \geq M^{-1} > 0$ in $\m{D}$ for all $I \in [N]$.
  We then scale \eqref{eq:lineqs} by $(V_{(I)}v)^{-1}$:
  \begin{equation}
    \label{eq:lineqs-scaled}
    (V_{(I)}v)^{-1} V_{(I)} f^I = (V_{(I)}v)^{-1} \calT_J^I f^J + (V_{(I)}v)^{-1} \Pi^I.
  \end{equation}
  Due to the consistent sign of $V_{(I)}v$, the fields $(V_{(I)}v)^{-1} V_{(I)}$ also have uniformly positive $t$-components and Cauchy hypersurface $\Sigma$.
  We use $v$ as a time coordinate in a Gr\"onwall argument.

  Let $\Gamma_s \coloneqq \{v = s\}$ denote the $s$-level set of $v$ for $s \in v(\m{D})$.
  Given $p \in \Gamma_s$ and $I \in [N]$, let $\tilde{\gamma}_{I,p}$ denote the future-directed integral curve of $(V_{(I)}v)^{-1} V_{(I)}$ such that $\tilde{\gamma}_{I,p}(s) = p$.
  Let $\sigma_{I,p} < s$ denote the unique parameter time at which $\ti{\gamma}_{I,p}(\sigma_{I,p}) \in \Sigma$.
  By design, $(v \circ \ti{\gamma}_{I,p})' = 1$ and $(v \circ \ti{\gamma}_{I,p})(s) = v(p) = s$, so in fact $(v \circ \ti{\gamma}_{I,p})(r) = r$ for all $r \in [\sigma_{I,p}, s]$.
  In other words, $\ti{\gamma}_{I,p}(r) \in \Gamma_r$ for all such $r$.
  In particular, $\sigma_{I,p} \geq \inf v$.

  Now if we integrate \eqref{eq:lineqs-scaled} along $\ti{\gamma_{I,p}}$, we obtain
  \begin{equation*}
    \abss{f^I(p)} \leq \abss{f^I \circ \tilde{\gamma}_{I,p}}(\sigma_{I,p}) + \int_{\sigma_{I,p}}^s \big[\abss{V_{(I)}v}^{-1}\big(\abss{\m{T}_J^I f^J} + \abss{\Pi^I}\big)\big] \circ \tilde{\gamma}_{I,p}(r) \d r.
  \end{equation*}
  Hence if we define $F(r) \coloneqq \max_J \sup_{\Gamma_r} \abss{f^J}$, we can write
  \begin{equation*}
    \abss{f^I(p)} \leq D_0 + M \int_{\sigma_{I,p}}^s [D_1 F(r) + D_2] \d r \leq D_0 + M \int_{\inf v}^s [D_1 F(r) + D_2] \d r.
  \end{equation*}
  Taking the supremum over $p \in \Gamma_s$ and $I \in [N]$, we obtain
  \begin{equation*}
    F(s) \leq D_0 + M \int_{\inf v}^s [D_1 F(r) + D_2] \d r \ForAll s \in v(\m{D}).
  \end{equation*}
  Applying Gr\"onwall and taking a supremum over $s \in v(\m{D})$, \eqref{eq:r1est} follows.
\end{proof}
In the case of extremal shocks, we apply this result with $v = u$, which is possible because $L_{(I')}u$ has a consistent sign.
When $I_0$ is intermediate, we use the result in a large region where $\mu$ is (very weakly) bounded from below.
In the complementary region very close to the shock, we can instead assume that certain coefficients are small in $L^1$, due to the smallness of the region.
We then apply the following:
\begin{proposition}
  \label{prop:r2est}
  Define $D_0$ as in Proposition~\ref{prop:r1est} as well as
  \begin{equation*}
    \kappa \coloneqq \max_I \sup_{p \in \calD} \sum_J \int_0^{s_{(I)} (p)} |\calT_J^I| \circ \gamma_{I,p} (s) \d s
  \end{equation*}
  and
  \begin{equation*}
    \beta D_0 \coloneqq \max_I \sup_{p \in \calD} \int_0^{s_{(I)} (p)} |\Pi^I| \circ \gamma_{I,p} (s) \d s.
  \end{equation*}
  Then if $\kappa < 1$,
  \begin{equation*}
    \max_I \sup_{\overline{\calD}} |f^I| \le \f {1 + \beta} {1 - \kappa} D_0.
  \end{equation*}
\end{proposition}
\begin{proof}
  Fix $\delta_0 > 0$.
  Let  $\calD'$ be the maximal globally hyperbolic subset of $\overline{\calD}$ with Cauchy hypersurface $\Sigma$ such that
  \begin{equation}
    \label{eq:L1-boot}
    \max_I \sup_{\calD'} |f^I| \le \zeta D_0
  \end{equation}
  for $\zeta \coloneqq \frac{1 + \beta}{1 - \kappa} + \delta_0$.
  This is clearly closed.
  We show it is open by improving on the bootstrap constant $\zeta$.

  Given $p \in \m{D}'$ and $I \in [N]$, we integrate \eqref{eq:lineqs} along $\gamma_{I,p}$ to obtain
  \begin{align*}
    \abss{f^I (p)} \leq D_0 + \int_0^{s_{(I)} (p)} (\abss{\m{T}_J^I} \abss{f^J} + \abss{\Pi^J}) \circ \gamma_{I,p}(s) \d s \leq (1 + \kappa \zeta + \beta) D_0.
  \end{align*}
  Rearranging $\zeta > \frac{1 + \beta}{1 - \kappa}$, we obtain $1 + \kappa \zeta + \beta < \zeta$.
  As this holds for all $I$, we have improved on our bootstrap assumption \eqref{eq:L1-boot}.
  
  By local existence and continuity, we can extend  \eqref{eq:L1-boot} to a neighborhood of $\m{D}'$.
  That is, $\m{D}'$ is open.
  Because $\m{D}$ is connected and $\m{D}'$ is nonempty by local existence, $\m{D}' = \m{D}$.
  Since $\delta_0 > 0$ is arbitrary, the proposition follows.
\end{proof}
We now introduce the following notation for the longest integration time we encounter in $\m{D}$:
\begin{equation}
  \label{eq:width}
  \calW(\calD,\Sigma) \coloneqq \max_I \sup_{p \in \m{D}} s_I(p).
\end{equation}
In our applications of Proposition~\ref{prop:r2est}, $\m{T}_J^I$ depends only on $\PhT$ and is thus order $1$.
We therefore satisfy the required $L^1$ smallness ($\al < 1$) once $\m{W}$ is sufficiently small.

\section{Proof of shock formation}
\label{sec:SF}

In this section, we prove Theorem~\ref{thm:SF}.
Most of the effort goes into the $(\tau, u)$-smoothness in part \ref{item:fund-smooth}.
Once this is established, parts \ref{item:shock-time} and \ref{item:MGHD} follow painlessly.

The structure of our argument depends on the extremality of the shock.
If $I_0$ is extremal, we can prove smoothness using a bootstrap and Proposition~\ref{prop:r1est}.
If $I_0$ is intermediate, we instead work in two regions: one causally closed region away from the shock, and its complement very near the shock.
In the first, $\mu$ has a uniform lower bound, and smoothness essentially follows from continuity in initial data.
In the second region, we use smallness to apply Proposition~\ref{prop:r2est}.

\subsection{Preliminaries}
Fix $k \geq 0$ representing the degree of regularity in Theorem~\ref{thm:SF}~\ref{item:fund-smooth}.
The background solution $\Theta$ is regular for $\tau \in [0, 1/10]$ and the initial perturbation is initially of order $\eps$.
Hence if we replace $\eps$ by $\eps/C$ for some $C(k) \gg 1$, continuous dependence on data implies that $\abss{\Gamma^\al \Php} \leq \eps$ for all $\abs{\al} \leq k$ and $\tau \in [0, 1/10]$.
We assume this throughout the section.
Strictly speaking this requires us to use $\eps/C$ in Theorem~\ref{thm:SF}, but we suppress this point.
It can be resolved by shrinking $\eps$ after the fact.

Next, Condition~\ref{cond:L1Th} provides a lower bound for $\mu$ where $u \not \in [u_1, u_2]$ with $u_2 - u_1 \leq 2 \eta$.
The essential properties of the eikonal function are unchanged if we shift it by a constant, so in the reminder of the section we center this interval about $u = 0$ and assume that $[u_1, u_2] = [-\eta, \eta]$.

Finally, our analysis involves various nonlinear function of $\psi$ such as $\lambda$ and $\xi$.
We work within a compact region of state space, so these functions are uniformly Lipschitz.
Hence we can control their perturbations in terms of the perturbation $\psip$ of $\psi$ itself.
We record this in the following lemma.
\begin{lemma}
  \label{lem:PGpest}
  Suppose $h = f \circ \psi - f \circ \Theta$ for some smooth $f$.
  Then $\abs{h} \leq C(A,\PhT) \abss{\psip}$.
\end{lemma}
Given these observations, we now show that $\psi$ remains regular in $(\tau, u)$, beginning with the more difficult intermediate-speed case.

\subsection{Eikonal regularity for intermediate shocks}

Away from the shock, we apply Gr\"onwall in a certain time coordinate $\taut$.
This begins from a spacelike Lipschitz curve $\Sigt_0$ constructed as follows.
Recall Definition~\ref{def:pert-func} and the constants $\eta,\delta_1$, and $\delta_2$ from Condition~\ref{cond:L1Th}.
Setting
\begin{equation*}
  \lambda^* \coloneqq \max_I \sup \abss{\hat{\lambda}_{(I)} - \lambda},
\end{equation*}
we let $\Sigt_0$ denote the following trapezoidal curve (see Figure~\ref{fig:twiddle-time}):
\begin{equation}
  \label{eq:taut}
  \tau =
  \begin{cases}
    0 & \text{if } \abs{u} \leq \eta,\\
    \frac{\delta_2}{1 + \lambda^*} \abs{u - \eta} & \text{if } \abs{u} \in (\eta, \eta + \delta_1],\\[3pt]
    \frac{\delta_1\delta_2}{1 + \lambda^*} & \text{if } \abs{u} > \eta + \delta_1.
  \end{cases}
\end{equation}
\begin{figure}[t]
  \centering
  \includegraphics[width = 0.7\linewidth]{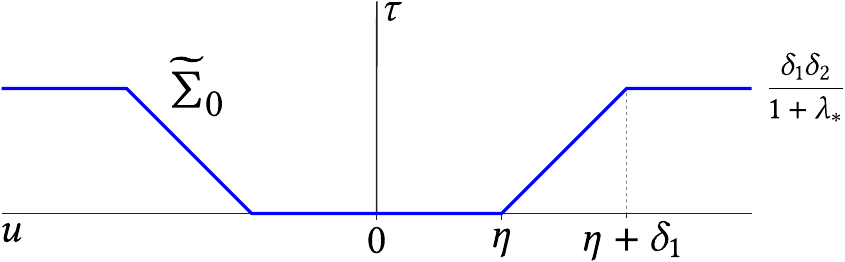}
  \caption{Profile of $\Sigt_0$ in eikonal coordinates.}
  \label{fig:twiddle-time}
\end{figure}
Because $\delta_2 \leq 1/10$, the slope $\frac{\delta_2}{1 + \lambda^*}$ ensures that $\Sigt_0$ is uniformly spacelike with respect to $\Theta$.

We now set $\taut = 0$ on $\Sigt_0$ and extend it via $\rd_\tau \taut = 1$.
Then $\taut$ is uniformly Lipschitz in $(\tau, u)$ and its level sets $\Sigt_s \coloneqq \{\taut = s\}$ are vertical translates of $\Sigt_0$.
By continuity, we may choose $\eps \ll 1$ so that $\psi$ is well behaved until $\Sigt_0$.
For simplicity of exposition, we assume that $\eps$ controls the data induced on $\Sigt_0$ as well as on $\{\tau = 0\}$.

\subsubsection{Away from shock formation}
We now prove estimates ``away from the shock'' where $0 \leq \taut \leq 1 - \delta$ for $\delta \coloneqq \delta_1\delta_2/[2(1 + \lambda^*)]$.
\begin{proposition}
  \label{prop:bs1}
  There exists $\eps_1 > 0$ such that for all $\eps < \eps_1$, the quantities $\Php$ exist in the globally hyperbolic region $\calDt_\delta \coloneqq \{0 \leq \taut \leq 1 - \delta\}$.
  Moreover, for all $k \in \N_0$, there exist $M_1,M_2 > 0$ such that
  \begin{equation*}
    \norms{\Php}_{\m{C}_{\tau,u}^k(\calDt_\delta)} \coloneqq \max_{h \in \Php} \norm{h}_{\m{C}_{\tau,u}^k(\calDt_\delta)} \leq M_2 \eps \exp\big[M_1 \big(1 + \delta^{-1} + \delta_2^{-1}\big)\big].
  \end{equation*}
\end{proposition}
\begin{proof}
  Let $s_0$ be the smallest value in $[\delta, 1]$ such that
  \begin{equation}
    \label{eq:far-boot}
    \norms{\Php}_{\m{C}_{\tau,u}^k(\calDt_{s_0})} \leq \eps^{3/4}.
  \end{equation}
  By improving on this bootstrap condition, we show that $s_0 = \delta$ provided $\eps \ll 1$.

  For the background wave, we have $\inf_{\calDt_{s_0}} \muT \ge \min\{\delta, \delta_2\}$.
  Thus by the bootstrap assumption \eqref{eq:far-boot},
  \begin{equation}
    \label{eq:mulower1}
    \inf_{\calDt_{s_0}} \mu \geq \min\{\delta, \delta_2\} - \eps^{3 / 4} \geq \f 1 2 \min\{\delta, \delta_2\} \quad \text{for } \eps \ll 1.
  \end{equation}
  This is the key observation: it implies that our equations are not singular in $\calDt_{s_0}$, so we can use routine hyperbolic estimates to propagate smallness.
  
  We wish to apply Proposition~\ref{prop:r1est} with $v = \taut$ to $\Gamma^\al \phi^{I'}$ for $\abs{\al} \leq k$.
  By Proposition~\ref{prop:schematiceqs}, these derivatives are governed by the operator $\mu L_{(I')} = \mu \partial_\tau + (\lambda_{(I)} - \lambda)\partial_u$.
  We therefore check that $\mu L_{(I')} \taut$ is uniformly positive, independent of $\eps$.
  To see this, we first observe that $\partial_\tau \taut = 1$.
  Hence where $\abs{u} \leq \eta$, \eqref{eq:taut} and \eqref{eq:mulower1} yield $\partial_u \taut = 0$ and $\mu L_{(I')} \taut = \mu \gtrsim 1$.
  Where $\abs{u} > \eta$, Condition~\ref{cond:L1Th} yields $\h\mu \geq \delta_2$ and hence by \eqref{eq:taut}, $\h \mu \h L_{(I')} \taut \geq \delta_2/(1 + \lambda^*)$.
  The bootstrap assumption then ensures that $\mu L_{(I')} \taut \gtrsim 1$.

  Using the equations from Proposition~\ref{prop:schematiceqs}, we can now apply Proposition~\ref{prop:r1est} with $v = \taut$, $V_{(I')} = \mu L_{(I')}$, and $f^{I'} = \Gamma^\al \phi^{I'}$.
  We have $M,D_1, \op{osc} v \lesssim 1$, while $D_0 \lesssim \eps$ and, by the bootstrap \eqref{eq:far-boot}, $D_2 \lesssim \eps^{3/2}$.
  Thus \eqref{eq:r1est} yields $\abss{\Gamma^\al \phi^{(I')}} \lesssim \eps$.
  Varying $\al$, we see that
  \begin{equation*}
    \norms{\phi^{I'}}_{\m{C}_{\tau,u}^k(\calDt_{s_0})} \lesssim \eps.
  \end{equation*}
  Integrating the remaining equations in Proposition~\ref{prop:schematiceqs} in $\tau$, we can successively conclude the same for $\Phsp$, $\psip$, and $\mup$.
  That is,
  \begin{equation*}
    \norms{\Php}_{\m{C}_{\tau,u}^k(\calDt_{s_0})} \lesssim \eps.
  \end{equation*}
  If $\eps$ is sufficiently small, this improves the bootstrap \eqref{eq:far-boot}, as desired.
\end{proof}

\subsubsection{Near shock formation}

We now control the fundamental unknowns in the region above $\Sigt_{1 - \delta}$ and below $\{\tau = 1 + \delta\}$.
To do so, we run a bootstrap argument in the quantity
\begin{equation}
  \label{eq:mus}
  \mus (p) \coloneqq \inf\big\{\mu(q) : q \in J^- (p) \cap \{ \taut \ge 1 - \delta \}\big\}
\end{equation}
Here $J^-(p)$ denotes the causal past of $p$, namely the set of past points that can be connected to $p$ via curves with slopes bounded between $\lambda_1$ and $\lambda_N$ (see Definition~\ref{def:causal}).
So $\mus$ is the infimum of $\mu$ in the portion of the causal past near the shock.
We let $\Sigs_s$ denote the level set $\{\mus = s\}$.
Then points in $\Sigs_0$ have a singularity ($\mu = 0$) in their causal past.
For this reason, $\mus$ is well-suited to the study of the boundary $\m{B}_{1 + \delta}$ of our MGHD.

To confine our attention near the shock, we truncate $\mus$ beyond a time slice.
We define
\begin{equation}
  \label{eq:muts}
  \muts (p) \coloneqq \min\{\mus(p), 1 + \delta - \tau(p)\}
\end{equation}
and write $\Sigts_s \coloneqq \{\muts = s\}$.
We solve our equation up to $\Sigts_0$, which is in fact the future boundary $\m{B}_{1 + \delta}$ of our MGHD in a box.
We have chosen parameters so that $\mu \to 0$ somewhere along this boundary, so it does include a singularity.
To understand $\Sigts_0$, we study $\Sigts_s$ as $s \searrow 0$.
\begin{figure}[t]
  \centering
  \includegraphics[width = 0.8\linewidth]{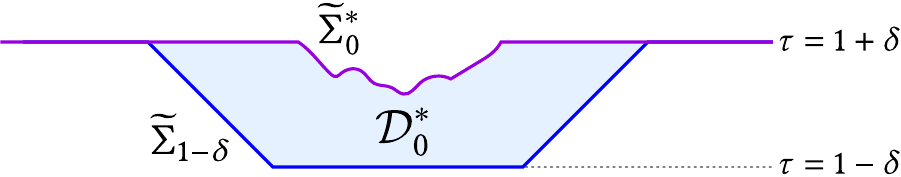}
  \caption{%
    The contour $\Sigts_0$ contains a portion of the future boundary of the MGHD (the irregular middle section above).
    We aim to show smoothness in $(\tau, u)$ in the region $\calDs_0$ between $\Sigt_{1 - \delta}$ and $\Sigts_0$.
  }
  \label{fig:near-shock}
\end{figure}

We now begin the bootstrap argument.
Using Proposition~\ref{prop:bs1}, we are free to assume that
\begin{equation*}
  \norms{\Php}_{\m{C}_{\tau,u}^k(\Sigt_{1 - \delta})} \leq \neweps
\end{equation*}
for some $\zeta \ll 1$ explicitly related to $\eps$.
Let $\calDs_s$ be the region above $\Sigt_{1 - \delta}$ and below $\Sigts_s$; for $\calDs_0$, see Figure~\ref{fig:near-shock}.
Let $s_0 > 0$ be the smallest value for which $\calDs_{s_0}$ is globally hyperbolic with Cauchy hypersurface $\Sigt_{1 - \delta}$ and
\begin{equation}
  \label{eq:near-boot}
  \norms{\Php}_{\m{C}_{\tau,u}^k(\calDs_{s_0})} \leq \neweps^{3/4}.
\end{equation}
Recalling \eqref{eq:width}, we let $\m{W} \coloneqq \m{W}(\calDs_{s_0}, \Sigt_{1 - \delta})$ with the width relative to the vector fields $L$ and $\mu L_{(I')}$.

We will improve on the bootstrap assumption \eqref{eq:near-boot} to conclude that it holds up to $\Sigts_0$.
Before doing so, we control the behavior of $\mus$ and $\muts$.
Let $\gamma_{(I)}^- (p)$ denote the past-directed integral curve of $L_{(I)}$ originating at $p$ and ending at $\Sigt_{1 - \delta}$.
\begin{lemma}
  \label{lem:musb}
  If $\zeta$ is sufficiently small, then for all $p \in \calDs_{s_0},$ $L \mu \leq -1/2$ and
  \begin{equation*}
    \mus (p) = \inf_{\gamma_{(1)}^- (p) \cup \gamma_{(N)}^- (p)} \mu.
  \end{equation*}
\end{lemma}
\begin{proof}
  Fix $p \in \calDs_{s_0}.$
  Condition~\ref{cond:L1Th} states that $\partial_\tau \h \mu \leq -3/4$ where $\abs{u} \leq \eta + \delta_1$, which includes the entire trapezoid $\{\taut \geq 1 - \delta\} \cap \{\tau \leq 1 + \delta\}$.
  Provided $\zeta \ll 1$, the bootstrap assumption \eqref{eq:near-boot} then ensures that $\partial_\tau \mu \leq -1/2$ in $J^-(p) \cap \{\taut \geq 1 - \delta\}$.
  In particular, $\mu$ is decreasing along the integral curves of $L$.
\end{proof}
In the following, let $\Lip_\Omega f$ denote the Lipschitz norm in the $(\tau,u)$ coordinate system of the function $f$ in the region $\Omega$.
\begin{lemma}
  \label{lem:LFlows}
  For $i \in \{1, 2\}$, let $\gamma_{(I')}^i(s) = \big(\tau^i(s), u^i(s)\big)$ be past-directed integral curves of $L_{(I')}$ in $\m{D}_{s_0}^*$ parametrized by $\dot{u}^i = \sgn(\lambda - \lambda_{(I')})$.
  Let $K \coloneqq \Lip_{\calDs_{s_0}} [\mu (\lambda_{(I')} - \lambda)^{-1}]$ and suppose $h \coloneqq \gamma_{(I')}^2(0) - \gamma_{(I')}^1(0)$ satisfies $\abs{h} \leq 1/10$.
  Then
  \begin{equation*}
    \abs{h} \e^{-\lambda^* \m{W} K} \leq \abss{\gamma_{(I')}^2 - \gamma_{(I')}^1} \leq \abs{h} \e^{\lambda^* \m{W} K}.
  \end{equation*}
  If in addition $u^1(0) = u^2(0)$, we have
  \begin{equation*}
    \abs{h} \e^{-\lambda^* \m{W} K} \leq \abss{\tau^2 - \tau^1} \leq \abs{h} \e^{\lambda^* \m{W} K}.
  \end{equation*}
\end{lemma}
\begin{proof}
  This is essentially the statement that trajectories of Lipschitz vector fields only converge or diverge exponentially.
  First suppose $u^1(0) = u^2(0)$, so by our parametrization $u^1 = u^2$.
  Then we can write
  \begin{equation*}
    \dot{\tau}^i(s) = -\mu \abss{\lambda_{(I')} - \lambda}^{-1}\big|_{(\tau^i(s),u^1(s))}.
  \end{equation*}
  It follows that
  \begin{equation*}
    \Big|\der{}{s}\log\frac{\abss{\tau^2 - \tau^1}}{\abs{h}}\Big| \leq \Lip_{\calDs_{s_0}} [\mu (\lambda_{(I')} - \lambda)^{-1}] \eqqcolon K.
  \end{equation*}
  Integrating from $s = 0$, we see that
  \begin{equation*}
    \Big|\log\frac{\abss{\tau^2 - \tau^1}}{\abs{h}}\Big| \leq K s.
  \end{equation*}
  The result follows from $s \le \lambda^* \calW$.
  To treat the general case, we take the same approach but instead study $\log(\abss{\gamma_{I'}^2 - \gamma_{(I')}^1}^2/\abs{h}^2)$.
  We omit the repeated details.
\end{proof}
In the following, inequalities involving derivatives of $\mus$ hold almost everywhere.
\begin{lemma}
  \label{lem:musmutsprop}
  In $\calDs_{s_0}$, $\mus$ is Lipschitz in $(\tau,u)$, $L_{(I)} \mus \le 0$ a.e. for all $I \in [N]$, and
  \begin{equation}
    \label{eq:mus-Lip}
    L \mus \le -\frac 1 2 \exp\big(\!-\lambda^* \calW \Lip_{\calDs_{s_0}} [\mu (\lambda_{(I')} - \lambda)^{-1}]\big).
  \end{equation}
  In particular, $\mus$ strictly decreases along integral curves of $L$.
  The same hold with $\muts$ in place of $\mus$.
\end{lemma}
\begin{proof}
  To see that $\mus$ is Lipschitz, we observe from Lemma~\ref{lem:musb} that we need only study the past extremal characteristics emanating from $p$ to evaluate $\mus(p)$.
  Lemma~\ref{lem:LFlows} ensures that nearby characteristics do not diverge too quickly, so $\mus$ inherits the Lipschitz property from $\mu$.

  Next, we note that $J^- (p) \subset J^- (q)$ whenever $q$ lies along a future directed causal curve starting at $p$.
  It follows that $\mus$ is nonincreasing along such curves, which include characteristics.
  So $L_{(I)} \mus \leq 0$.

  Now let $\sigma_h$ denote the shift operation in $\tau$ by $h$, so $\sigma_h(\tau, u) = (\tau + h, u)$.
  By Lemma~\ref{lem:musb}, $\mus(p)$ is attained at some point $q$ along a past extremal characteristic through $p$.
  Given $h > 0$, Lemma~\ref{lem:LFlows} ensures that the corresponding characteristic through $p' \coloneqq \sigma_h p$ shifts up in $\tau$ by at least $h' \coloneqq h \e^{-\lambda^* \m{W} K}$.
  Hence $q' \coloneqq \sigma_{h'}q$ lies in the causal past of $p'$.
  Now, Lemma~\ref{lem:musb} states that $\partial_\tau \mu \leq -1/2$, so
  \begin{equation*}
    \mus(\sigma_h p) = \mus(p') \leq \mu(q') \leq \mu(q) - \frac{h'}{2} = \mus(p) - \frac{h'}{2}.
  \end{equation*}
  Since this holds for all $h > 0$ sufficiently small, we conclude that
  \begin{equation*}
    L \mus = \partial_\tau \mus \leq -\frac{h'}{2h} = - \frac{1}{2} \e^{-\lambda^* \m{W} K}.
  \end{equation*}
\end{proof}
These results imply the following.
\begin{corollary}
  If the bootstrap \eqref{eq:near-boot} holds in $\calDs_s$, then this region is globally hyperbolic with Cauchy hypersurface $\Sigt_{1 - \delta}$.
\end{corollary}
\begin{proof}
  By Lemma~\ref{lem:musmutsprop}, $\muts$ has non-vanishing gradient and is nonincreasing along every future-directed causal curve.
  Thus if we intersect its superlevel set $\{\muts \geq s\}$ with the causal future $\{\taut \geq 1 - \delta\}$ of a spacelike curve, the result $\calDs_s$ is globally hyperbolic with Cauchy hypersurface $\{\taut = 1 - \delta\} = \Sigt_{1-\delta}$.
\end{proof}
Next, we show that the width $\calW = \calW(\calDs_{s_0},\Sigt_{1 - \delta})$ is rather small:
\begin{lemma}
  \label{lem:widthest}
  For $\neweps$ sufficiently small, we have $\calW \ls \eta + \delta_1.$
\end{lemma}
\begin{proof}
  The $\tau$ extent of $\calDs_{s_0}$ is at most $2\delta = \delta_1\delta_2/(1 + \lambda^*) < \delta_1$.
  Because $L \tau = 1$, it follows that any integral curve of $L$ starting in $\calDs_{s_0}$ will reach $\Sigt_{1 - \delta}$ within parameter time $\delta_1$.
  By \eqref{eq:taut} and the definition of $\calDs_{s_0}$, its $u$ extent it at most $2(\eta + \delta_1)$.
  Using Lemma~\ref{lem:PGpest} and the bootstrap assumption, it follows that any integral curve of $\mu L_{(I')}$ starting in $\calDs_{s_0}$ will reach $\Sigt_{1 - \delta}$ within parameter time
  \begin{equation*}
    2(\eta + \delta_1) \max_{(I')} \sup_{\calDs_{s_0}} \abss{\hat{\lambda}_{(I')} - \hat{\lambda}}^{-1} + C \neweps^{3 / 4}.
  \end{equation*}
  The conclusion follows from \eqref{eq:width}.
\end{proof}
We can now improve on the bootstrap \eqref{eq:near-boot}.
\begin{proposition}
  \label{prop:r2bsrec}
  For $\eta, \delta_1$ and $\neweps$ sufficiently small, we have
  \begin{equation*}
     \norms{\Php}_{\m{C}_{\tau,u}^k(\calDs_{s_0})} \lesssim \neweps.
  \end{equation*}
\end{proposition}
\begin{proof}
  We induct on regularity.
  Suppose we have shown
  \begin{equation}
    \label{eq:inductive-hyp}
    \norms{\Php}_{\m{C}_{\tau,u}^{\ell - 1}(\calDs_{s_0})} \lesssim \neweps
  \end{equation}
  for some $0 \leq \ell \leq k$, with the convention that \eqref{eq:inductive-hyp} is vacuous when $\ell = 0$.
  Take $\al$ such that $\abs{\al} = \ell$.
  Drawing on Proposition~\ref{prop:schematiceqs}, we wish to apply Proposition~\ref{prop:r2est} to the equations for $\Gamma^\al \phi^{I'}$.
  By Lemma~\ref{lem:widthest}, this is possible provided $\eta, \delta_1 \ll 1$, for then the $L^1$ quantity $\kappa$ is small.
  Using \eqref{eq:near-boot}, we conclude that $\abss{\Gamma^\al \phi^{I'}} \lesssim \neweps$.
  Varying $\al$,
  \begin{equation*}
    \norms{\phi^{I'}}_{\m{C}_{\tau,u}^\ell(\calDs_{s_0})} \lesssim \neweps.
  \end{equation*}
  Integrating the remaining equations in Proposition~\ref{prop:schematiceqs} in $\tau$, we successively conclude the same for $\Phsp$, $\psip$, and $\mup$.
  That is, $\norms{\Php}_{\m{C}_{\tau,u}^\ell(\calDs_{s_0})} \lesssim \neweps.$
  The proposition follows from induction.
\end{proof}
Having improved on the bootstrap, we conclude by continuity that it holds on $\m{D}_0^*$, that is, up to $\Sigts_0$.
This completes the proof that the state variables remain as smooth as the data as $\muts \rightarrow 0$.
It also shows that the perturbative unknowns remain of size $\eps$ throughout the course of evolution.

\subsubsection{Choice of parameters}
We have introduced a number of parameters in the above argument, constrained in various ways.
We now describe how these can be chosen to satisfy the necessary bounds.

We first choose a simple wave satisfying Condition~\ref{cond:L1Th} with $\eta$ sufficiently small that we can use Proposition~\ref{prop:r2est} in the proof of Proposition~\ref{prop:r2bsrec}.
The same proof requires $\delta_1 \ll 1$, which can be arranged automatically.
We likewise obtain $\delta_2 \in (0, 1/10]$.
These parameters determine $\delta \coloneqq \delta_1 \delta_2/[2(1 + \lambda^*)]$.
We then take $\neweps$ sufficiently small that the proof of Proposition~\ref{prop:r2bsrec} carries through.
Recalling Proposition~\ref{prop:bs1}, we finally take $\eps < \eps_1$ such that $M_2 \eps \exp[M_1 (1 + z^{-1} + \delta_2^{-1})] \leq \neweps$.

\subsection{Eikonal regularity for extremal shocks}
We first note that outside of a compact region of spacetime, the data are trivial, meaning that we have the desired estimates by finite speed of propagation.
With this in mind, we immediately begin a bootstrap argument.
We define variants of $\mus$ and $\muts$ from \eqref{eq:mus} and \eqref{eq:muts}, preserving the notation.
Let
\begin{align*}
  \mus (p) &\coloneqq \inf\big\{\mu(q) \colon q \in J^- (p) \cap \{ \tau \ge 0 \}\big\},\\
  \muts (p) &\coloneqq \min\{\mus(p), 1 + 1/10 - \tau (p)\}.
\end{align*}
Given $s > 0$, let $\calD_s$ be the region above $\{\tau = 0\}$ and below $\{\mus = s\}$.
Let $s_0$ be the largest value of $s$ for which
\begin{equation}
  \label{eq:extremal-boot}
  \norms{\Php}_{\m{C}_{\tau,u}^k(\calDs_{s_0})} \leq \eps^{3/4}.
\end{equation}
To improve on this bootstrap, we first control $\Phns^{I'}$ and then $\Phsp$, $\psip$, and $\mup$.
Because $I_0$ is extremal, $\mu L_{(I')} u = \lambda_{(I')} - \lambda$ has a consistent sign across all $I'$ and is uniformly bounded away from zero.
We can therefore apply Proposition~\ref{prop:r1est} to Proposition~\ref{prop:schematiceqs} with $V_{(I')} = \mu L_{(I')}$ and $v = u$.
Using the bootstrap \eqref{eq:extremal-boot}, we conclude that $\abss{\Gamma^\al \phi^{I'}} \lesssim \eps$ for all $\abs{\al} \leq k$.
That is, $\norms{\phi^{I'}}_{\m{C}_{\tau,u}^k(\calDs_{s_0})} \lesssim \eps.$
Integrating the remaining equations in $\tau$, we can successively show the same for $\Phsp$, $\psip$, and $\mup$.
Hence
\begin{equation}
  \label{eq:extremal-improved}
  \norms{\Php}_{\m{C}_{\tau,u}^k(\calDs_{s_0})} \lesssim \eps.
\end{equation}
This improves on the bootstrap \eqref{eq:extremal-boot} once $\eps$ is sufficiently small, so by continuity \eqref{eq:extremal-improved} holds in $\calD_0$.
This completes the proof of Theorem~\ref{thm:SF}~\ref{item:fund-smooth}.

\subsection{Proof of shock formation}

We can now integrate the equation for $\mu$ to show that a shock forms.
This argument has appeared several times in the literature (see, for example, \cite{LukSpe02}), so we only sketch it.

Because the background wave $\Theta$ first forms a shock at time $1$, Proposition~\ref{prop:sw} ensures that $\partial_\tau \h\mu = -1$ along its shocking characteristic.
We have just shown that $\mup$ is $\eps$-small in $\m{C}^1$, so in a neighborhood of the $\Theta$-shocking characteristic, $\partial_\tau \mu = -1 + \m{O}(\eps)$.
Integrating from $\mu|_{\tau = 0} = 1$, we see that $\mu \searrow 0$ at a time $t_* = 1 + \m{O}(\eps)$.
Moreover, $\inf \abss{(\partial_u \Theta)^{I_0}} > 0$ on a neighborhood of its shocking characteristic, while $\psip$ is $\eps$-small in $\m{C}^1$ in the $(\tau,u)$ coordinate system.
Hence $\inf \abss{(\partial_u \psi)^{I_0}} > 0$ where $\mu \searrow 0$.
Since $\partial_u \psi = \mu \partial_x \psi,$ we conclude that $\abss{(\partial_x \psi)^{I_0}} \nearrow \infty$ as $\mu \searrow 0$.
That is, a spatial derivative of the solution blows up.
Meanwhile, $\psi = \Theta + \m{O}(\eps)$ remains bounded, so this singularity is a shock.
This completes the proof of Theorem~\ref{thm:SF}~\ref{item:shock-time}.

\subsection{Higher regularity from $\m{C}^1$ smallness}
We performed the calculations above for arbitrary $k \in \N_0$.
However, it is only necessary to control the quantities in $\Php$ to prove shock formation.
These involve first derivatives of $\psi$, so we need only ensure that the initial perturbation $v$ in Theorem~\ref{thm:SF} is small in $\m{C}^1$.
Given such smallness, we can propagate higher regularity.
\begin{corollary}
  \label{cor:HigherEsts}
  In the setting of Theorem~\ref{thm:SF}, we have
  \begin{equation*}
    \norms{\Php}_{\m{C}^k (\calDs_0)} \le M \norm{v}_{\m{C}^{k+1}},
  \end{equation*}
  for a constant $M$ depending on $k,\Theta,A,\eta,\delta_1,\delta_2,\delta,$ and $\norms{\Php}_{\m{C}^{k - 1} (\calDs_0)}$.
\end{corollary}
\noindent
This can be compared with the third statement in Theorem~\ref{thm:SF}.
\begin{proof}
  The propagation of higher regularity estimates after closing a nonlinear problem has been carried through in several contexts, so we omit the details.
  The corollary follows inductively from a slight modification of calculations above.
  In particular, we note from Proposition~\ref{prop:schematiceqs} that the analysis of the linearized equations from Section~\ref{sec:HyperbolicEsts} carries through even after commutation.
\end{proof}

\subsection{Maximal globally hyperbolic development} \label{sec:MGHD}
We now show that $\calB_{1 + \delta} = \Sigts_0$ is Lipschitz.
In the $(\tau,u)$ plane, the level sets $\Sigts_s$ of $\muts$ are Lipschitz graphs over the $u$ axis converging to a Lipschitz $\Sigts_0$.
After all, $\muts$ is uniformly Lipschitz and by \eqref{eq:mus-Lip}, $\rd_\tau \muts$ is uniformly negative in a neighborhood of the nontrivial portion of $\calB_{1 + \delta}$.
Moreover, the fundamental unknowns extend smoothly to functions on $\Sigts_0$.
It follows that $x$ extends smoothly as a function of $(\tau,u)$ to $\Sigts_0$ as well.
By the definition \eqref{eq:taut} of $\tau$, the inextensible portions of $\calB_{1 + \delta}$ are confined to a small region with $\abs{\tau - 1e} \leq \delta$ and $\abs{u} \leq \eta + \delta_1$.
When we map this to the $(t, x)$ plane, the smoothness of $x$ implies that the image of the inextensible portions lie within a ball $B$ that can be made arbitrarily small by shrinking $\eta$, $\delta_1$, and $\delta$.

We now use Condition~\ref{cond:MGHDstab}.
Using the smallness of $B$, we are free to assume that $F|_B = 0$, as we can otherwise perform a constant shear to reduce to this case.
Then there exists $c > 0$ such that $\lambda_{(1)} \le -c$ and $\lambda_{(N)} \ge c$ in $B$.
For every $q \in \Sigts_s \cap B$, the truncated cone $\{\abs{x - x(q)}/c < t(q) - t\} \cap B$ lies in the causal past of $q$, and hence is disjoint from $\Sigts_s$.
It follows that $\Sigts_s$ is a Lipschitz graph with uniform constant for $0 < s < 1 / 10$.
Moreover, as $s \rightarrow 0$, $\Sigts_s$ converges to the image of $\Sigts_0$ in the $(\tau,u)$ plane under the mapping $(\tau,u) \rightarrow (t,x)$.
This image $\Sigts_0$ is itself a Lipschitz graph because the Lipschitz constant of $\Sigts_s$ in $(t, x)$ is uniformly controlled as $s \rightarrow 0$.

Next, we explicitly characterize the components of $\calB_{1 + \delta}$ in the $(\tau,u)$ plane.
Let
\begin{equation*}
  \m{B}_{1 + \delta}^{\text{pre}} = \Sigts_0 \cap \{ \mu = 0 \} \cap \{ \rd_u \mu = 0 \}.
\end{equation*}
This corresponds to points on $\Sigts_0$ where $\mu$ vanishes and where the vector fields $\mu L_{(I')} = \mu \partial_\tau + (\lambda_{(I')} - \lambda)\partial_u$ are tangent to $\{ \mu = 0 \}$.
We then set
\begin{equation*}
  \m{B}_{1 + \delta}^{\text{sing}} = \Sigts_0 \cap \{ \mu = 0 \} \cap \{ \rd_u \mu \ne 0 \}.
\end{equation*}
We claim this singular set is empty unless $\lambda$ is extremal.
To see this, suppose to the contrary that $\lambda$ is intermediate and there exists $p \in \m{B}_{1 + \delta}^{\text{sing}}$.
We may assume $\rd_u \mu(p) > 0$, as a symmetric argument treats the other sign.
Then by continuity and Lemma~\ref{lem:musb}, there exists $c > 0$ and a neighborhood $B$ of $p$ in which
\begin{equation}
  \label{eq:slopes}
  \mu L_{(N)} \mu = \mu \partial_\tau\mu + (\lambda_{(N)} - \lambda)\partial_u \mu \geq c \And L \mu \leq -c.
\end{equation}
Here we use the fact that $\lambda$ is intermediate to produce a faster eigenvalue $\lambda_{(N)}$.
In the following, we choose a parameter $r > 0$ sufficiently small that we remain within $B$.
Following an integral curve of $L$ backward from $p$ by parameter $r$, we arrive at a point $q \in \Sigts_s$ for some $s > 0$, and $\mu(q) \asymp r$.
If we now follow an integral curve of $\mu L_{(N)}$ backward from $q$, \eqref{eq:slopes} implies that $\mu$ will fall to zero within time $\asymp r$.
Since $\mu$ vanishes in the causal past of $q$, $\mu^*(q) = 0$, contradicting $q \in \Sigts_s$.
So indeed $\m{B}_{1 + \delta}^{\text{sing}}$ is empty if $\lambda$ is intermediate.
For generic singularities, one can check that these definitions of $\m{B}^{\text{pre}}$ and $\m{B}^{\text{sing}}$ coincide with the informal descriptions in Section~\ref{sec:SFstatements}.

Continuing, we set
\begin{equation*}
  \m{B}_{1 + \delta}^{\text{Cau}} = \Sigts_0 \cap \{ \mu \ne 0 \} \cap \{ p : J^{-1} (p) \cap \{ \mu = 0 \} \neq \emptyset\}.
\end{equation*}
This is the set of points where the solution is still regular because $\mu \ne 0$, but where a singularity occurs somewhere in the causal past.
Finally, we set
\begin{equation*}
  \m{B}_{1 + \delta}^{\text{ext}} = \Sigts_0 \setminus (\m{B}_{1 + \delta}^{\text{pre}} \cup \m{B}_{1 + \delta}^{\text{sing}} \cup \m{B}_{1 + \delta}^{\text{Cau}}).
\end{equation*}
These correspond to regular points without singularities in their causal past.
They are only included in $\Sigts_0$ because of our cutoff at $t = 1 + \delta$.
We can hyperbolically extend the solution in a neighborhood of such points.
For example, the extensible points in $\m{B}_{1 + \delta / 2}$ lie in the interior of $\m{M}_{1 + \delta}$.
Finally, Theorem~\ref{thm:SF} involves $\m{B}_{t_* + \delta}$ rather than $\m{B}_{1 + \delta}$.
Since $t_* = 1 + \m{O}(\eps)$, we can apply the above results to $\m{B}_{t_* + \delta}$ by halving $\delta$ and reducing $\eps$ to be sufficiently small relative to $\delta$.

\subsection{A perverse MGHD}
We now turn to Proposition~\ref{prop:Bcurvature}.
Consider the decoupled system $\psi = (v, w)$ with
\begin{equation*}
  \rd_t v + v \rd_x v = 0 \And \rd_t w - \rd_x w = 0,
\end{equation*}
consisting of the Burgers equation and linear transport.

We choose smooth data $0 \leq v(0, \anon) \leq 2$ such that $\rd_x v(0, \anon)$ attains its global minimum of $-1$ nondegenerately at $x = \frac{1}{n}$ for all $n \in \N$, and $v(0, \frac{1}{n}) = 1$.
The method of characteristics shows that $v$ is smooth for $t < 1$ and develops preshocks at time $1$ and positions $x_n = 1 + \frac{1}{n}$ and $1$.
Locally near $(1, x_n)$, $v$ can be smoothly extended below the integral curves of $\partial_t + v \partial_x$ and $\partial_t - \partial_x$.
Global hyperbolicity implies $\m{M}$ cannot extend above these curves, so in fact they form the future boundary $\m{B}$ of $\m{M}$ near $(1, x_n)$.
Thus the curvature of the boundary includes the point-mass $2\delta_{(1, x_n)}$ for each $n \in \N$.
Because the sequence $(1, x_n)_{n \in \N}$ has a limit point, the curvature of $\m{B}_{1 + \delta}$ is not a Radon measure no matter the value of $\delta > 0$.
This proves Proposition~\ref{prop:Bcurvature}.

\section{A homogeneous expansion}
\label{sec:homogeneous}

Thanks to Theorem~\ref{thm:SF}, we have propagated regularity in $(\tau,u)$ up to a nontrivial portion of the zero set of $\mu$.
In this section, we use this information to show that a generic preshock admits an expansion in certain homogeneous functions.
This expansion has a universal leading term, which plays a central role in our analysis of the inviscid limit.
We consider perturbations around simple waves that are nondegenerate in the following sense:
\begin{condition}
  \label{cond:genericity}
  A simple wave $\Theta$ is \emph{strongly nondegenerate} if $L \h\mu = \rd_u \h\lambda \hspace{-0.5pt}=\hspace{-0.5pt} \f {\rd \lambda} {\rd r_{I_0}} (\rd_u \Theta)^{I_0}$ has a unique minimum, and there $\partial_u^3 \h\lambda > 0$.
\end{condition}
In this setting, the maximal development in Theorem~\ref{thm:SF} has the generic form described in Section~\ref{sec:SFstatements}.
In particular, $\m{B}_{t_* + \delta}^{\text{pre}}$ consists of a single preshock; for the behavior of the Cauchy horizon and singular set, we direct the reader to Section~\ref{sec:SFstatements} as well as \cite{Christodoulou01}, \cite{AbbSpe}, \cite{AbbSpe23}, and \cite{ShkVic_2024}.
For the remainder of the paper, we set aside the full maximal development and instead focus on the spacetime slab $[0, T] \times \R$ between the initial data and the time of the preshock.
We will describe the solution $\psi$ in great detail in the vicinity of this preshock.
As our analysis is local, we could readily treat data with several simultaneous isolated preshocks, but we work around a single preshock for simplicity.

We show that $\psi$ admits an expansion in certain functions $\fu$ and $\fm$ that approximate $u$ and $\mu^{-1}$ near the preshock and scale in $t$ and $x$ in a natural way.
We defer their definitions to Section~\ref{sec:rhomexp}, and here state the result informally:
\begin{theorem}
  \label{thm:hom-informal}
  In the setting of Theorem~\ref{thm:SF}, assume in addition that $\Theta$ is strongly nondegenerate in the sense of Condition~\ref{cond:genericity}.
  Then for $\eps$ sufficiently small, $\psi$ admits an expansion near the preshock in $t, \fu$, and $\fm$.
\end{theorem}
\noindent
For a rigorous statement, see Theorem~\ref{thm:psi-hom-exp} below.

To show this result, we observe that Theorem~\ref{thm:SF}~\ref{item:fund-smooth} implies that $\psi$ and $\mu$ admit Taylor series in $(\tau, u)$ up to an order determined by the regularity of the data.
Because $\mu$ governs the change of variables from $(t, x)$ to $(\tau, u)$, we can thereby expand $x$ in $(\tau, u)$.
Recall that $T$ denotes the time of first shock.
Using Condition~\ref{cond:genericity}, we show that this expansion takes the form
\begin{equation}
  \label{eq:x-Taylor}
  x = a_0(\tau - T)(u - u_0) + b_0 (u - u_0)^3 + \ldots
\end{equation}
for constants $a_0, b_0$, and $u_0$.
The leading terms in \eqref{eq:x-Taylor} allow us to identify homogeneous approximations $\fu$ and $\fm$ of $u$ and $\mu^{-1}$.
We show that $u$ admits an expansion in $\tau$, $\fu$, and $\fm$, and this leads in turn to an expansion for $\psi$.

\subsection{Eikonal expansion for \texorpdfstring{$x$}{x}}

In \cite{CG23}, the second and third authors used the hodograph method to expand $x$ in $\tau$ and $u$ near a preshock in the Burgers equation.
This method is confined to scalar conservation laws, and thus applies to the simple wave $\Theta$ (which is effectively scalar) but not to the full solution $\psi$.
Nonetheless, we expect $\psi$ to exhibit the same behavior up to perturbations of order $\eps$.
To succinctly express our results, we introduce schematic notation for the relatively smooth errors encountered in the construction of this expansion.
\begin{notation}
  Given $k \in \N_0$, we let $E_k$ denote a $\m{C}^k$ function of $(\tau, u)$ that may change from line to line satisfying $\norm{E_k}_{\m{C}^k} \lesssim_{k,A,\Theta} 1$.
  Let $E_\infty$ denote a function of the form $E_k$ for all $k \in \N_0$.
\end{notation}
By Theorem~\ref{thm:SF}, the fundamental unknowns $\Ph$ admit Taylor expansions up to order $k$ in the $(\tau,u)$ coordinate system.
We now closely examine $\mu$ and $x$ as functions of $(\tau,u)$ near the preshock.
\begin{proposition}
  \label{prop:expansion1}
  In the setting of Theorem~\ref{thm:SF}, assume Condition~\ref{cond:genericity} and that $\psi(0, \anon) \in \m{C}^{k+1}$ for some $k \geq 3$.
  Then there exists $\eps'(\Theta) > 0$ such that if $\norm{v}_{\m{C}^3} \leq \eps'$:
  \begin{enumerate}[label = \textup{(\roman*)}, itemsep = 2pt]
  \item
    \label{item:unique-preshock}
    There exists a unique point $(t_*, u_0)$ with the smallest $\tau$ coordinate in the zero set of $\mu$.
    It lies in $\m{B}_{t_* + \delta}^{\textnormal{pre}}$.
    
  \item
    \label{item:mu-exp}
    Let $\gamma(\tau) \coloneqq \argmin \mu(\tau, \anon)$ for $\tau \geq \tfrac{1}{2}$.
    Then $\gamma(\tau) = u_0 + (t_* - t) E_{k-1}(\tau)$ and there exist constants $a_0,b_0 > 0$ such that
    \begin{equation}
      \label{eq:mu-exp}
      \mu(\tau, u) = a(\tau) (t_* - \tau) + 3 b(\tau) [u - \gamma(\tau)]^2 + [u - \gamma(\tau)]^3 E_{k - 3}(\tau, u)
    \end{equation}
    with $a(\tau) = a_0 + (t_* - \tau) E_{k - 1}(\tau)$ and $b(\tau) = b_0 + (t_* - \tau) E_{k - 3}(\tau)$.
   
 \item
   \label{item:x-exp}
   For the same functions $a(\tau)$ and $b(\tau)$, we have
   \begin{equation}
     \label{eq:x-in-tau-u}
     \begin{aligned}
       x(\tau,u) - x\big(\tau,\gamma(\tau)\big) = a(\tau) (t_* - \tau) [u - \gamma(\tau)&] + b(\tau) [u - \gamma(\tau)]^3\\
                                                                                      &+ [u - \gamma(\tau)]^4 E_{k - 3} (\tau,u).
     \end{aligned}
   \end{equation}
 \end{enumerate}
\end{proposition}
\noindent
The constant $\eps'$ must be sufficiently small that Theorem~\ref{thm:SF} can be applied and the perturbation $v$ does not disrupt the strong nondegeneracy of $\Theta$ from Condition~\ref{cond:genericity}.
For this reason, $\eps'$ controls the third derivative of $v$.
If $\Theta$ is just barely nondegenerate in the sense that $0 < \partial_u^3 \h\lambda \ll 1$, then $\eps'$ must be correspondingly small.
\begin{proof}
  We begin with the structure of the background simple wave $\Theta$.
  By Proposition~\ref{prop:sw}~\ref{item:simple}, $\partial_\tau \muT$ is constant in time.
  By hypothesis, $\Theta$ first shocks at time $1$, so Proposition~\ref{prop:sw}~\ref{item:first} implies that $\min \partial_\tau \muT = -1$.
  Condition~\ref{cond:genericity} states that this minimum is nondegenerate and achieved at a single point $\h u \in \R$.
  Recalling that $\Theta \in \m{C}^\infty$, it follows that for $c_0 \coloneqq \partial_{\tau u u} \h \mu(\h u) > 0$,
  \begin{equation}
    \label{eq:background-nondeg}
    \partial_\tau \h \mu = -1 + \tfrac{c_0}{2}(u - \h u)^2 + (u - \h u)^3 E_\infty(u), \quad \partial_{\tau u} \muT = c_0(u - \h u) + (u - \h u)^2 E_\infty(u).
  \end{equation}
  
  Turning to the full solution, we find
  \begin{equation}
    \label{eq:mixed}
    \partial_{\tau u} \mu = \partial_{\tau u} \muT + \partial_{\tau u} \mup = c_0(u - \h u) + (u - \h u)^2 E_\infty(u) + \partial_{\tau u} \mup.
  \end{equation}
  When $\tau = 0$, $\partial_u \mu = \partial_u \h \mu = \partial_u \mup = 0$.
  We can thus integrate from $\tau = 0$ to find
  \begin{equation}
    \label{eq:muu}
    \partial_u \mu(\tau, u) = c_0 \tau (u - \h u) + \tau (u - \h u)^2 E_\infty(u) + \partial_u \mup.
  \end{equation}
  In the remainder of the proof, assume $\tau \geq 1/2$ and $\norm{\mup}_{\m{C}^2} \leq \eta$ for a parameter $\eta(\Theta) > 0$ that we may shrink from line to line.
  Let $z \coloneqq \norm{\partial_u \mup}_\infty \leq \eta$.
  If $\eta$ is sufficiently small, the linear term in \eqref{eq:muu} dominates the quadratic, and we find
  \begin{equation*}
    \pm \partial_u \mu(\tau, \h u \pm 4 c_0^{-1}z) \geq \frac{z}{2} \geq 0.
  \end{equation*}
  By the intermediate value theorem, there exists a root $\gamma(\tau)$ of $\partial_u \mu(\tau, \anon)$ such that
  \begin{equation}
    \label{eq:gamma-bd}
    \gamma - \h u \lesssim \norm{\mup}_{\m{C}^1}.
  \end{equation}
  Differentiating \eqref{eq:muu}, \eqref{eq:gamma-bd} yields
  \begin{equation*}
    \partial_{uu} \mu(\tau, \gamma) = c_0 \tau + (\gamma - \h u) E_\infty(\gamma) + \partial_{uu} \mup(\tau, \gamma) = c_0 \tau + \m{O}(\norm{\mup}_{\m{C}^2}).
  \end{equation*}
  Perhaps after shrinking $\eta$, there thus exists $r > 0$ such that
  \begin{equation}
    \label{eq:convex}
    \partial_{uu}\mu(\tau, u) \geq \frac{c_0 \tau}{4} > 0 \quad \text{if } \abs{u - \gamma} \leq r.
  \end{equation}

  Next, we write $\mu = \h \mu + \bar{\mu}$.
  Recalling that $\h \mu|_{\tau = 0} \equiv 1$ and $\partial_\tau \h \mu$ is constant, we have
  \begin{equation*}
    \mu = 1 + \tau \partial_\tau \h \mu + \mup.
  \end{equation*}
  Now
  \begin{equation*}
    \inf_{\abs{u - \h u} \geq r/2} \partial_\tau \h \mu > \inf \partial_\tau \h \mu = -1.
  \end{equation*}
  Hence if $\eta$ is sufficiently small, $\mu(\tau, \anon)$ achieves its global minimum within $r/2$ of $\h u$.
  Using \eqref{eq:gamma-bd}, we can arrange $\abs{\gamma - \h u} \leq r/2$ if $\eta \ll 1$.
  Then \eqref{eq:convex} implies that $\mu(\tau, \anon)$ is uniformly convex within $r/2$ of $\h u$.
  It follows that the global minimum of $\mu(\tau, \anon)$ is unique, and must be $\gamma(\tau)$.
  That is, $\gamma(\tau) = \argmin \mu(\tau, \anon)$.
  This shows \ref{item:unique-preshock}, for the earliest root $u_0 = \gamma(t_*)$ of $\mu$ is unique.

  Now, $\gamma$ is differentiable by the implicit function theorem.
  Differentiating $\partial_u \mu(\tau, \gamma) = 0$ in $\tau$, we see that
  \begin{equation*}
    \gamma' = - \frac{\partial_{\tau u}\mu}{\partial_{uu}\mu}\Big|_{(\tau, \gamma)}.
  \end{equation*}
  The denominator is bounded away from zero by \eqref{eq:convex}.
  For the numerator, \eqref{eq:mixed}, \eqref{eq:gamma-bd}, and Proposition~\ref{prop:schematiceqs} yield
  \begin{equation*}
    \partial_{\tau u} \mu(\tau, \gamma) \lesssim \abs{\gamma - \h u} + \abs{\partial_{\tau u} \mup} \lesssim \norm{\partial_u \mup}_\infty + \norm{\partial_{\tau u} \mup}_\infty \lesssim \norms{\Php}_{\m{C}^1}.
  \end{equation*}
  Hence $\norm{\gamma}_{\m{C}^1} \lesssim \norms{\Php}_{\m{C}^1}.$
  Expanding expressions for higher derivatives, we see that $\norm{\gamma}_{\m{C}^k} \lesssim \norms{\Php}_{\m{C}^k} < \infty$ by Theorem~\ref{thm:SF}.
  Thus by Taylor's remainder theorem,
  \begin{equation*}
    \gamma(\tau) = u_0 + (t_* - \tau) E_{k-1}(\tau).
  \end{equation*}

  Now let $b(\tau) \coloneqq \tfrac{1}{6}\partial_{uu} \mu\big(\tau, \gamma(\tau)\big) \geq \tfrac{c_0\tau}{24}$, where we have used \eqref{eq:convex}.
  Then Theorem~\ref{thm:SF} yields $\norm{b}_{\m{C}^{k-2}} \lesssim \norm{\mu}_{\m{C}^k} < \infty$.
  Setting $b_0 \coloneqq b(t_*) > 0$, Taylor's remainder theorem yields
  \begin{equation*}
    b(\tau) = b_0 + (t_* - \tau) E_{k - 3}(\tau).
  \end{equation*}
  Moreover, applying the remainder theorem in $u$ about $\gamma$ to $\mu$, we find
  \begin{equation*}
    \mu(\tau, u) = m(\tau) + 3 b(\tau) (u - \gamma)^2 + (u - \gamma)^3 E_{k-3}(\tau, u)
  \end{equation*}
  for $m(\tau) \coloneqq \mu\big(\tau, \gamma(\tau)\big)$.
  Noting that $m(t_*) = 0$, let $a(\tau)$ satisfy $m(\tau) = (t_* - \tau) a(\tau)$.
  Differentiating, \eqref{eq:background-nondeg} and \eqref{eq:gamma-bd} yield
  \begin{equation*}
    m' = \partial_\tau \mu(\tau, \gamma) = \partial_\tau \h \mu + \partial_\tau \mup = -1 + \m{O}(\abs{\gamma - \h u}^2 + \abs{\partial_\tau \mup}) = -1 + \m{O}(\norm{\mup}_{\m{C}^1}).
  \end{equation*}
  Thus if $\eta$ is sufficiently small, $a_0 \coloneqq m'(t_*) > 0$.
  Moreover, Proposition~\ref{prop:schematiceqs} and \eqref{eq:gamma-bd} imply that
  \begin{equation*}
    m'' = \partial_{\tau\tau} \mu + \gamma' \partial_{\tau u} \mu = \partial_{\tau \tau} \mup + \gamma'(\partial_{\tau u}\h \mu + \partial_{\tau u} \mup) \lesssim \norms{\Php}_{\m{C}^1}.
  \end{equation*}
  Similar reasoning implies that $\norm{m}_{\m{C}^{k+1}} \lesssim \norms{\Php}_{\m{C}^k} < \infty$, so the remainder theorem yields
  \begin{equation*}
    m(\tau) = a(\tau)(t_* - \tau) \And a(\tau) = a_0 + (t_* - \tau) E_{k-1}(\tau).
  \end{equation*}
  We have now verified every component of \ref{item:mu-exp}.
  Next, \eqref{eq:mu} implies $\partial_u x = \mu$.
  Integrating \eqref{eq:mu-exp} in $u$ from $\gamma$, we obtain \ref{item:x-exp}.

  Finally, we have assumed throughout that $\norm{\mup}_{\m{C}^2} \leq \eta(\Theta)$.
  By Corollary~\ref{cor:HigherEsts}, this holds provided $\norm{v}_{\m{C}^3} \leq \eps'(\Theta)$ for sufficiently small $\eps' > 0$.
\end{proof}

\subsection{A convenient gauge}

We now have an expression \eqref{eq:x-in-tau-u} for $x$ in terms of $\tau = t$ and $u$.
At a high level, this resembles the prototypical example of a preshock in the Burgers equation: the cubic cusp $\fu$ satisfying
\begin{equation}
  \label{eq:fu}
  x = - a_0 t \fu + b_0 \fu^3.
\end{equation}
We now use gauge symmetry to bring \eqref{eq:x-in-tau-u} into close alignment with \eqref{eq:fu}.

To begin, we observe that the class of hyperbolic conservation laws is preserved under Galilean transformations.
That is, if we set $\bar{t} \coloneqq t + t_0$ and $\bar{x} \coloneqq x - vt - x_0$ for $t_0,x_0,v \in \R$, then a quick calculation shows that \eqref{eq:psi} is equivalent to
\begin{equation}
  \label{eq:psi-gauge}
  \partial_{\bar t} \psi + \bar{A}(\psi) \partial_{\bar x} \psi = 0
\end{equation}
for $\bar{A} \coloneqq A - v \op{Id}$.
The operation $A \mapsto \bar{A}$ shifts all eigenvalues by the constant $v$ and preserves the eigenvectors.
Thus $\bar{A}$ inherits strictly hyperbolicity and the genuine nonlinearity in Condition~\ref{cond:gennon} from $A$.
It follows that our entire hyperbolic analysis of \eqref{eq:psi} applies to \eqref{eq:psi-gauge} as well.

The free parameters $(t_0,x_0)$ and $v$ represent spacetime shifts and spatial shear, respectively.
We choose these to simplify \eqref{eq:x-in-tau-u}.
Let $(-t_0,x_0) \coloneqq \big(t_*, x(t_*, u_0)\big)$ be the spacetime location of the preshock in Proposition~\ref{prop:expansion1} and let $v \coloneqq \lambda \circ \psi(t_0, x_0)$ be the speed of the shocking characteristic at the preshock.
After our Galilean transformation, we have shifted to a moving frame in which the preshock occurs at the spacetime origin with zero shocking speed.
Due to our choice of temporal sign, time now begins at $t_0 = - T < 0$.

We are free to make two further choices of ``labeling.''
Shifting the data of the eikonal function $u$ by $u_0$, we can assume that $u(0,0) = 0.$
Similarly, composing $\bar{A}$ with a shift in $\m{V}$, we can replace $\psi$ by $\psi - \psi(0,0)$, so that $\psi(0,0) = 0$.
In sum:
\begin{proposition}
  In the setting of Proposition~\ref{prop:expansion1}, we can use gauge freedoms to arrange the following:
  \begin{enumerate}[label = \textup{(\roman*)}, itemsep = 2pt]
  \item The preshock forms at $(0,0)$ in both the $(t,x)$ and $(\tau,u)$ coordinate systems.
  \item The shocking eigenvalue vanishes at the preshock: $\lambda(0,0) = 0$.
  \item The state variables vanish at the preshock: $\psi(0,0) = 0$.
  \end{enumerate}
\end{proposition}
We now consider Proposition~\ref{prop:expansion1} in this gauge.
We have arranged
\begin{equation*}
  t_* = u_0 = \gamma(0) = x\big(0, \gamma(0)\big) = 0.
\end{equation*}
Using \eqref{eq:der-x}, $(\partial_\tau x) (0,0) = \lambda(0,0) = 0$.
And by Theorem~\ref{thm:SF}, $\partial_\tau^2 x = \partial_\tau \lambda = E_k$.
Hence by Taylor's remainder theorem,
\begin{equation}
  \label{eq:x-gamma}
  x\big(\tau,\gamma(\tau)\big) = \tau^2 E_k (\tau).
\end{equation}
We use this in our analysis of \eqref{eq:x-in-tau-u} below.

\subsection{Homogeneous functions}

Because $\partial_x u = \mu^{-1}$ blows up at the preshock, $u$ does not admit a traditional power series in $(t, x)$ about the origin.
However, we show that $u$ does admit an expansion in functions resembling powers of $\fu$ from \eqref{eq:fu}.
The terms in this expansion are determined by their scaling properties, which generalize those of $\fu$.
We recall the following terminology from~\cite{CG23}.
\begin{definition}
  Given $r \in \R$, a function $f : \R_{\le 0} \times \R \rightarrow \R$ is \emph{$r$-homogeneous} if
  \begin{equation*}
    f(\lambda^2 t,\lambda^3 x) = \lambda^r f(t,x) \ForAll \lambda > 0, t \leq 0, \text{ and } x \in \R.
  \end{equation*}
\end{definition}
\noindent
Using \eqref{eq:fu}, we can directly verify that $\fu$ is $1$-homogeneous.
We will show that $u$ admits an expansion in functions of integral homogeneity, with leading term $\cub$.
To this end, we make frequent use of a $-2$-homogeneous function $\cubder$ that approximates $\partial_xu = \mu^{-1}$:
\begin{equation*}
  \fm \coloneqq \partial_x \fu = \f 1 {a_0 \abs{t} + 3 b_0 \fu^2}.
\end{equation*}
Here we have used $t < 0$ to write $-t$ as $\abs{t}$.

Following \cite{CG23}, it is convenient to introduce the 1-homogeneous function
\begin{equation*}
  \fd \coloneqq \fm^{-\f 1 2} = (a_0 \abs{t} + 3 b_0 \fu^2)^{\f 1 2},
\end{equation*}
which approximates $\sqrt{\mu}$.
This is nonnegative and vanishes precisely at the preshock $(0, 0)$.
We use it as a homogeneous measure of proximity to the preshock.
For example, if $f$ is $r$-homogeneous and continuous away from the origin,
\begin{equation*}
  \abs{f(t, x)} = \dist^r \abss{f\left(\dist^{-2} t, \dist^{-3}x\right)} \leq \dist^r \sup_{\dist = 1} \abs{f} \lesssim \dist^r.
\end{equation*}
For this reason, we frequently use powers of $\dist$ to control error terms.

Our expansions will not involve arbitrary homogeneous functions.
Rather, we will only need polynomials in $t$, $\cub$, and $\cubder$.
We therefore collect a few useful identities relating these building blocks.
\begin{lemma}
  \label{lem:identities}
  We have $t = -a_0^{-1} \fm^{-1} + 3 a_0^{-1}b_0 \fu^2$,  $\rd_x \fu = \fm$, and $\rd_t \fu = a_0 \fu \rd_x \fu = a_0 \fu \fm$.
\end{lemma}
\begin{proof}
  Direct calculation and the chain rule.
\end{proof}
This lemma exemplifies a broader principle: differentiation in $t$ and $x$ reduces homogeneity by $2$ and $3$, respectively.
Thus $x$ derivatives ``cost more'' than $t$ derivatives, and tend to make expressions more singular.

We also recall Lemma~2.1 of \cite{CG23}, which we express in the following form:
\begin{lemma}
  \label{lem:fuest}
  In $\R_{\leq 0} \times \R$, we have
  \begin{enumerate}[label = \textup{(\roman*)}, itemsep = 2pt]
  \item $|\fu| \asymp |x|^{\f 1 3}$ for $\abs{t} \ls |x|^{\f 2 3}$ and $|\fu| \asymp |x| |t|^{-1}$ for $|x|^{\f 2 3} \ls \abs{t}$.
  \item $\fm \asymp (|t| + |\fu|^2)^{-1}$ and $\fm \ls |\fu| |x|^{-1}$.
  \item $\fd \asymp |t|^{\f 1 2}$ for $|x|^{\f 2 3} \ls \abs{t}$ and $\fd \asymp |x|^{\f 1 3}$ for $|x|^{\f 2 3} \gs \abs{t}$.
  \end{enumerate}
\end{lemma}
Our homogeneous expansions will only be relevant near the preshock, so we are free to restrict our analysis to a neighborhood of the origin.
We will choose a small parameter $\kappa > 0$ and work in the $(t, x)$ region
\begin{equation}
  \label{eq:expansion-region}
  \D \coloneqq [-\kappa, 0] \times [-\kappa, \kappa].
\end{equation}

\subsection{Homogeneous expansion for the eikonal function}
\label{sec:rhomexp}

We now develop a homogeneous expansion for $u$, broadly following the strategy of \cite{CG23}.
We first show that $u = \fu + \m{O}(\fd^2)$, and then linearize around $\cub$ to construct higher order homogeneous correctors.
The proof of the former has a somewhat different character, as the linearization about $\cub$ is singular, and we thus require relatively strong bounds on $u - \cub$ to make use of it.
We begin by showing such bounds.
Throughout, we assume $u$ satisfies the hypotheses of Proposition~\ref{prop:expansion1} for some $k \geq 3$.

To show $u \approx \cub,$ we observe that $u$ and $\fu$ are related to $x$ through similar equations.
Indeed, using Proposition~\ref{prop:expansion1} and \eqref{eq:x-gamma}, we can write
\begin{equation}
  \label{eq:uxfu}
  a_0 \abs{t} u + b_0 u^3 + F(t,u) = x = a_0 \abs{t} \fu + b_0 \fu^3,
\end{equation}
for
\begin{equation}
  \label{eq:ETS}
  F = t^2 E_{k-3} + tu^2 E_{k-3} + u^2 E_{k - 3}.
\end{equation}
We next show that both $u$ and $\cub$ are small near the preshock.
\begin{lemma}
  In the region $\D$ defined in \eqref{eq:expansion-region}, $|u|, \abs{\cub} \lesssim \kappa^{\f 1 3}$ if $\kappa$ is sufficiently small.
\end{lemma}
\begin{proof}
  For $\cub$, this is immediate from Lemma~\ref{lem:fuest}.
  Therefore consider $u$.

  Let $K \coloneqq \sup \abs{\lambda} < \infty$.
  When $t = 0$, \eqref{eq:mu-exp} implies that $\partial_u x = \mu = 3b_0 u^2 + \m{O}(u^3)$.
  Integrating in $u$, we see that we can choose $\kappa$ sufficiently small that $\abs{x} \geq b_0\abs{u}^3/2$ for all $\abs{x} \leq (K + 1) \kappa$.
  Rearranging, $\abs{u} \lesssim \abs{x}^{1/3} \lesssim \kappa^{1/3}$.

  Now recall that $u$ is conserved along the $I_0$ characteristic curves.
  For any $(t, x) \in \D$, the characteristic through $(t, x)$ travels at speed $\abs{\lambda} \leq K$ and thus terminates at a location $(0, y)$ satisfying $\abs{y} \leq \abs{x} + K \kappa \leq (K + 1) \kappa$.
  By our work above, we have $\abs{u(t, x)} = \abs{u(0, y)} \lesssim \kappa^{1/3}$.
\end{proof}
We bootstrap this into a stronger bound on $u$.
\begin{lemma}
  \label{lem:uest}
  In $\D$, $|u| \ls |\fu| + |t|$ and $|\mu| \ls \dist^2$ if $\kappa$ is sufficiently small.
\end{lemma}
\begin{proof}
  For the moment, suppose $\abs{t} \leq \abs{u}/C$ for some $C > 0$.
  Provided $C$ is sufficiently large and $\kappa$ sufficiently small, \eqref{eq:ETS} yields $\abs{F} \leq \tfrac{1}{2}\abs{u}(a_0 \abs{t} + b_0 u^2)$.
  It follows from \eqref{eq:uxfu} and the reverse triangle inequality that
  \begin{equation}
    \label{eq:cubic-bound}
    \abs{x} \asymp \abs{u}(\abs{t} + u^2).
  \end{equation}
  We consider two cases depending on which term dominates the right side of \eqref{eq:cubic-bound}.
  If $a_0 \abs{t} \leq b_0 u^2$, then \eqref{eq:cubic-bound} yields $\abs{u} \asymp \abs{x}^{1/3}$.
  Then $\abs{t} \lesssim u^2 \asymp \abs{x}^{2/3}$, so by Lemma~\ref{lem:fuest}, $\abs{u} \asymp \abs{x}^{1/3} \asymp \abs{\cub}$.
  If on the other hand $b_0 u^2 \leq a_0 \abs{t}$, then \eqref{eq:cubic-bound} implies that $\abs{u} \asymp \abs{x} \abs{t}^{-1}$, so in turn $\abs{x}^2 \abs{t}^{-2} \asymp \abs{u}^2 \lesssim \abs{t}$.
  Rearranging, we see that $\abs{x}^{2/3} \lesssim \abs{t}$.
  Hence by Lemma~\ref{lem:fuest}, $\abs{u} \asymp \abs{x} \abs{t}^{-1} \asymp \abs{\cub}$.

  In either case, we have $\abs{u} \asymp \abs{\cub}$.
  This all holds provided $\abs{u} \geq C \abs{t}$, so in general we have $\abs{u} \lesssim \abs{\cub} + \abs{t}$.
  
  For $\mu$, we can write \eqref{eq:mu-exp} as
  \begin{equation}
    \label{eq:mu-err}
    \mu = a_0 \abs{t} + 3 b_0 u^2 + t^2 E_{k-3} + tu E_{k-3} + u^3 E_{k-3}.
  \end{equation}
  If $\kappa$ is sufficiently small, the first two terms dominate and $\mu \asymp \abs{t} + u^2$.
  Then our bound on $u$ implies that $\mu \lesssim \abs{t} + \cub^2 \lesssim \dist^2$.
\end{proof}
We can now determine the leading term in $u$.
\begin{proposition}
  \label{prop:usolve}
  In $\D$, $u = \fu + \m{O}(\fd^2)$ and $\rd_x u = \mu^{-1} = \fm + \m{O}(\fd^{-1})$.
  In general, for all $m + n \leq k - 2$,
  \begin{equation}
    \label{eq:usolve-higher}
    \rd_t^m \rd_x^n u = \rd_t^m \rd_x^n \fu + \m{O}(\fd^{- 2 m - 3 n + 2}).
  \end{equation}
\end{proposition}
\begin{proof}
  For each $t \geq 0$, $u(t, \anon)$ is a bijection $\R \to \R$.
  Let $u^{\circ - 1}$ denote its inverse in the spatial variable and define $h(t, x) \coloneqq x - u^{\circ -1}\big(t, u(t, x)\big)$, so $\cub(t, x - h) = u(t, x)$.
  Using \eqref{eq:uxfu}, we obtain
  \begin{equation*}
    x - h = a_0 \abs{t} \cub(t, x - h) + b_0 \cub^3(t, x - h) = a_0 \abs{t} u + b_0 u^3 = x - F.
  \end{equation*}
  That is, $h(t, x) = F\big(u(t, x)\big)$.
  Combining \eqref{eq:ETS} and Lemma~\ref{lem:uest}, we find $\abs{h} \lesssim \dist^4$.

  Now, the mean value theorem implies that
  \begin{equation}
    \label{eq:MVT}
    u(t,x) = \fu\big(t,x + h(t,x)\big) = \fu(t,x) + \rd_x \fu\big(t,x + \zeta(t,x)\big) h(t, x)
  \end{equation}
  for some $\zeta$ between $0$ and $h$ (which may be negative).

  We again treat two cases.
  Suppose $\abs{t} \lesssim \abs{x}^{2/3}$, in which case Lemma~\ref{lem:fuest} yields $h \lesssim \abs{x}^{4/3}$.
  Then $\abs{\zeta} \leq \abs{h} \ll \abs{x}$, so $\abs{t} \lesssim \abs{x}^{2/3} \asymp \abs{x + \zeta}^{2/3}$.
  Using Lemma~\ref{lem:fuest} again, we find
  \begin{equation*}
    \partial_x\cub(t, x + \zeta) = \cubder(t, x + \zeta) \asymp \abs{x}^{-2/3} \asymp \dist^{-2}.
  \end{equation*}
  It follows that $\abss{\rd_x \fu\big(t,x + \zeta) h} \lesssim \dist^2$.

  Conversely, if $\abs{x}^{2/3} \lesssim \abs{t}$, then Lemma~\ref{lem:fuest} implies that $h \lesssim \abs{t}^2$.
  On the other hand, $\partial_x \cub = (a_0\abs{t} + b_0 \cub^2)^{-1} \lesssim \abs{t}^{-1}$, so in this case $\abss{\rd_x \fu\big(t,x + \zeta) h} \lesssim \abs{t} \asymp \dist^2$.
  Having treated each case, we conclude from \eqref{eq:MVT} that $u = \cub + \m{O}(\dist^2)$.
  
  For $\partial_x u = \mu^{-1}$, we observe that \eqref{eq:mu-err} and our bounds on $u$ yield $\mu = \cubder^{-1} + \m{O}(\dist^3)$.
  Inverting, we find $\partial_xu = \cubder + \m{O}(\dist)$.
  Then $\partial_t u = -\lambda \partial_x u$ allows us to treat $\partial_t u$, and we can inductively apply higher derivatives to these identities to obtain \eqref{eq:usolve-higher}.
\end{proof}
Having found the leading homogeneous part of $u$, we can now iteratively solve for higher order corrections.
These will be polynomials in $t,$ $\cub$, and $\cubder$.
We illustrate the procedure by finding the second term in the expansion.
Writing $v \coloneqq u - \cub$, \eqref{eq:uxfu} yields
\begin{equation}
  \label{eq:u-corrector}
  a_0 \abs{t} v + 3 b_0 \fu^2 v + 3 b_0 \fu v^2 + b_0 v^3 + F(t,u) = 0.
\end{equation}
By Proposition~\ref{prop:usolve}, $a_0 \abs{t} v + 3 b_0 \fu^2 v + 3 b_0 \fu v^2 = \cubder^{-1} v + \m{O}(\dist^3)$.
Moreover, using the proposition in \eqref{eq:ETS}, we see that there exist constants $c_{2,0},c_{1,2},c_{0,4} \in \R$ such that $F(t, u) = c_{2,0}t^2 + c_{1,2}t \cub^2 + c_{0,4}\cub^4 + \m{O}(\dist^5)$.
We can thus rearrange \eqref{eq:u-corrector} to find
\begin{equation*}
  v = -\cubder  (c_{2,0}t^2 + c_{1,2}t \cub^2 + c_{0,4}\cub^4) + \m{O}(\dist^3).
\end{equation*}
The main term of $v$ is a polynomial in $(t, \cub, \cubder)$ of homogeneity $2$.
Once again, by iteratively differentiating, we can show that this relation holds in a $\m{C}^{k-2}$ sense.
Repeating this process to higher order, we obtain a homogeneous expansion for $u$.
At each stage, we effectively apply Taylor's remainder theorem to the functions $E_{k-3}$ in $F$ to extract its next homogeneous part.
We can do so $k - 3$ times, which corresponds to an expansion to homogeneity $k - 2$.
\begin{proposition}
  \label{prop:eikonal-hom}
  For every $h \in [k-2]$, there is a polynomial $q_h$ in $(t,\cub,\cubder)$ with terms of homogeneity at most $h$ such that for every $m,n \in \N_0$ with $h + m + n \le k - 2$,
  \begin{equation*}
    \rd_t^m \rd_x^n u = \rd_t^m \rd_x^n q_h + \m{O}(\fd^{h - 2 k - 3 n + 1}) \quad \text{in } [t_0, 0) \times \R.
  \end{equation*}
  Moreover, $q_1 = \cub$.
\end{proposition}
Now recall that Theorem~\ref{thm:SF} ensures that the fundamental unknowns $\Ph$ admit Taylor series in $(\tau, u)$ to order $k$.
Substituting the homogeneous expansion for $u$ into these series, we find analogous expansions for all the fundamental unknowns, and in particular $\psi$.
This completes the proof of our informal Theorem~\ref{thm:hom-informal}.

In the following, let $\partial_{I_0} \lambda$ denote the derivative of $\lambda$ in direction $r_{I_0}$ in $\m{V}$.
By \eqref{eq:psixrhs}, $\partial_{I_0} \lambda = -\xi_{I_0 I_0}^{I_0}$.
This derivative is nonzero by the genuine nonlinearity in Condition~\ref{cond:gennon}.
We also recall the constant $a_0 > 0$ from Proposition~\ref{prop:expansion1}~\ref{item:mu-exp}.
\begin{theorem}
  \label{thm:psi-hom-exp}
  Let $\psi$ satisfy the hypotheses of Proposition~\ref{prop:expansion1} for some $k \geq 3$.
  Then for every $h \in [k-2]$, there is a polynomial $p_h$ in $(t,\cub,\cubder)$ with terms of homogeneity at most $h$ such that for every $m,n \in \N_0$ with $h + m + n \le k - 2$, we have
  \begin{equation*}
    \rd_t^m \rd_x^n \psi = \rd_t^m \rd_x^n p_h + \m{O}(\fd^{h - 2 m - 3 n + 1}) \quad \text{in } [t_0, 0) \times \R.
  \end{equation*}
  Moreover, $p_1 = -a_0 \partial_{I_0} \lambda(0)^{-1} \cub r_{I_0}(0)$.
\end{theorem}
\begin{proof}
  By Theorem~\ref{thm:SF}~\ref{item:fund-smooth}, $\psi$ is $\m{C}^k$ in $(\tau, u)$.
  Thus the expansion $p_h$ follows from composition with Proposition~\ref{prop:eikonal-hom}.
  The same proposition implies that $p_1 = \partial_u \psi(0) \cub$, so it remains to compute $\partial_u \psi(0).$

  This is a vector, and we first show that its nonshocking components are zero.
  Indeed, Theorem~\ref{thm:SF}~\ref{item:fund-smooth} implies that $(\partial_x \psi)^{I'} = \mu^{-1}(\partial_u \psi)^{I'} \lesssim 1$.
  Multiplying by $\mu$ and using $\mu(0) = 0$, we see that $(\partial_u \psi)^{I'}(0) = 0$.
  So
  \begin{equation}
    \label{eq:parallel}
    \partial_u \psi(0) = (\partial_u \psi)^{I_0}(0) r_{I_0}(0).
  \end{equation}

  Because $\lambda(0) = 0$, we have $\lambda(\psi) = \nab \lambda(0) \cdot \psi + \m{O}(\abs{\psi}^2)$ near the origin.
  By \eqref{eq:parallel}, the leading homogeneous part is $\partial_{I_0} \lambda(0) (\partial_u \psi)^{I_0}(0) \cub$, where $\partial_{I_0}$ indicates differentiation in $\m{V}$ in direction $r_{I_0}(0)$.
  Using Lemma~\ref{lem:identities} and Proposition~\ref{prop:eikonal-hom}, we can extract the leading $-1$-homogeneous part of \eqref{eq:eikonal}:
  \begin{equation*}
    a_0 \cub \cubder + \partial_{I_0} \lambda(0) (\partial_u \psi)^{I_0}(0) \cub \cubder = 0.
  \end{equation*}
  It follows that
  \begin{equation*}
    (\partial_u \psi)^{I_0}(0) = -\frac{a_0}{\partial_{I_0} \lambda(0)}.
  \end{equation*}
  Combined with $p_1 = \partial_u \psi(0) \cub$ and \eqref{eq:parallel}, the theorem follows.
\end{proof}

\appendix

\section{Causality and hyperbolic developments}
\label{sec:appendix1}

In this appendix, we define the notions of causality and domain of dependence used in the present paper.
The definitions we provide are less general and phrased differently than usual, but they capture the spirit of the standard definitions and allow us to provide a self-contained treatment.
We refer the reader to \cite{CouHil89}, \cite{Rau12}, and \cite{Rin13} for further details.

Let $\s{V} \coloneqq \{V_1, \dots, V_N\}$ be a set of Lipschitz vector fields on an open subset of $\R^{1 + 1}$.
We view these as the characteristic fields of a hyperbolic system in $\R^{1 + 1}$.
We make the simplifying assumption that the $\partial_t$ component of $V_i$ is positive and independent of $i$, while the $\partial_x$ components are strictly ordered by $i$.
That is,
\begin{equation*}
  V_i = f(t, x) \partial_t + h_i(t, x) \partial_x
\end{equation*}
for $f > 0$ and $h_1 < \ldots < h_N$.
Suppose also that $h_1 < 0 < h_N$.
While convenient, these assumptions are not necessary.
What really matters is the direction of $V_i$, so one could rescale each by potentially different positive factors without changing anything significant.
The following notions are implicitly defined relative to $\s{V}$.
\begin{definition}
  Given $p \in \R^{1+1}$, let $\calC_p$ denote the double cone in $T_p \R^{1 + 1} \setminus \{0\}$ determined by $V_1$ and $V_N$.
  Let $\calC_p^\pm$ denote the cone corresponding to $\pm V_1$ and $\pm V_N$.
\end{definition}
The cone $\m{C}_p$ separates timelike and spacelike vectors.
Precisely:
\begin{definition}
  The set $\calT_p^+$ of future-directed timelike vectors is the connected component of $\rd_t$ in $T_p \R^{1 + 1} \setminus (\m{C}_p \cup \{0\})$.
  The past-directed timelike vectors $\calT_p^-$ are defined analogously with $-\rd_t$ in place of $\rd_t$.
  Let $\calT_p \coloneqq \calT_p^+ \cup \calT_p^-$ denote the set of timelike vectors.
  We say a vector is spacelike if it is in neither $\m{C}_p$ nor $\m{T}_p$.
\end{definition}
\noindent
Roughly, curves that are nowhere spacelike model the propagation of signals.

The following definitions are simpler for $\m{C}^1$ curves, but we allow for Lipschitz curves because they appear in the present paper.
We thus consider tangent cones rather than tangent vectors.
The only possibly quite irregular object we consider is the boundary of the maximal development.
Everything else is at least piecewise $\m{C}^1$, so the reader can keep this situation in mind for the following definitions.
\begin{definition}
  A causal curve is any Lipschitz curve whose tangent cones are contained in $\calC_p \cup \calT_p$ everywhere.
  It is future-directed if increasing its parameter increases $t$, and past-directed otherwise.
\end{definition}
Meanwhile, spacelike curves provide a snapshot of the state of the system.
\begin{definition}
  A Lipschitz curve $\Sigma$ is a Cauchy hypersurface if its tangent cones are spacelike everywhere.
\end{definition}
We can now define regions on which we can solve the system of hyperbolic equations corresponding to $\s{V}$.
\begin{definition}
  A domain $\Omega \subset \R^{1+1}$ is globally hyperbolic with Cauchy hypersurface $\Sigma$ if, for any point $p \in \Omega$, the integral curve of any $V_i$ through $p$ intersects $\Sigma$ exactly once while remaining within $\Omega \cup \Sigma$.
\end{definition}
Intuitively, we imagine prescribing initial data on $\Sigma$.
Then $\Omega$ is globally hyperbolic if we can trace any point in $\Omega$ back to the data along a causal curve.
This notion is common in general relativity, where the curves are geodesics of a Lorentzian metric (see \cite{Rin13}).
More generally, this definition can be extended to hyperbolic operators in higher dimensions by replacing integral curves by bicharacteristic curves associated to the principal symbol of the operator.

Suppose we are given a hyperbolic PDE with principal part $\s{V}$ and initial data $\Psi_0$ on some Cauchy hypersurface $\Sigma$.
Assume we have a solution $\Psi$ to the PDE on a globally hyperbolic domain $\Omega \subset \R^{1 + 1}$ with Cauchy hypersurface $\Sigma$, and $\Psi|_\Sigma = \Psi_0$.
Then we say $(\Psi,\Omega)$ is a \emph{development} of the initial data $(\Psi_0,\Sigma)$.
\begin{definition}
  We say $(\Psi,\Omega)$ is a maximal globally hyperbolic development of $(\Psi_0,\Sigma)$ if there does not exist another development $(\Psi_1, \Omega_1)$ of $(\Psi_0,\Sigma)$ such that $\Omega \subset \Omega_1$ and $\Psi_1|_\Omega = \Psi$.
\end{definition}
Maximal globally hyperbolic developments are the ``largest'' developments possible.
Intuitively, we can extend a development so long as a singularity does not form.
Precisely, we can extend $(\Psi, \Omega)$ provided there is a point $p \in \partial \Omega \setminus \Sigma$ for which $\Psi$ is nonsingular at $p$ and in its causal past.
For more information on MGHDs in the context of the compressible Euler equations, we direct the reader to \cite{Christodoulou01,AbbSpe,AbbSpe23,ShkVic_2024}.

It is natural to ask whether maximal globally hyperbolic developments are unique.
If the solution satisfies uniform bounds on a given (not necessarily maximal) development, classical theory ensures its uniqueness.
However, uniqueness becomes delicate in the presence of singularities, and we refer the reader to \cite{EpReSb19} for more information on this fascinating topic.
Because singularities are often unavoidable in hyperbolic equations, grappling with this issue is a necessary part of understanding the uniqueness of maximal globally hyperbolic developments.
We refer the reader to forthcoming work of Abbrescia, Blue, Sbierski, and Speck providing constructions of provably unique MGHDs. We emphasize that we have not dealt with the issue of uniqueness in the present work, and we leave this as a direction for future investigation.

We close with a few more notions tied to causality.
\begin{definition}
  \label{def:causal}
  Given $p \in \R^{1+1}$, let $J^- (p)$ (resp. $J^+(p)$) denote the set of points $q$ that can be connected to $p$ by a Lipschitz causal curve that is past-directed (resp. future-directed) starting at $p$ and ending at $q$.
\end{definition}
We call $J^- (p)$ the causal past of $p$ and $J^+ (p)$ the causal future of $p$.
Intuitively, $J^- (p)$ is the set of points that can influence what happens at $p$, and $J^+ (p)$ is the set of points that $p$ can influence.

\printbibliography
\end{document}